\documentclass[journal]{IEEEtran}
\usepackage[hyphens]{url}
\usepackage{amsmath}
\usepackage{dsfont}
\usepackage{mathrsfs}
\usepackage{amsfonts}
\usepackage{amssymb}
\usepackage{amsthm}
\usepackage{graphicx}
\usepackage{color}
\usepackage{epsfig} 
\usepackage{xcolor}
\usepackage{bbm}
\usepackage{arydshln}
\usepackage{multirow}
\usepackage{subcaption}
\newcommand{\tabspacer}{\rule[0pt]{0pt}{8pt}}
\usepackage[ruled,vlined]{algorithm2e}
\usepackage{float}
\usepackage{nccmath} 
\usepackage[normalem]{ulem}
\usepackage{mathtools}
\makeatletter
\newcommand{\ltxlabel}{\ltx@label}
\makeatother

\usepackage[
    style=ieee,
    doi=false,
    isbn=false,
    url=true,
    eprint=false,
    backend=bibtex,
    maxcitenames=2,
    mincitenames=1,
    natbib=true
    ]{biblatex}

\bibliography{collection}

\newcommand{\T}{\intercal} 

\newtheorem{theorem}{Theorem}[section]

\newtheorem{prop}[theorem]{Proposition}

\newtheorem{asmp}[theorem]{Assumption}

\begin{document}
\title{Resilient Control of Platooning Networked\\ Robotic Systems via Dynamic Watermarking}

\author{
Matthew~Porter,
Arnav~Joshi,
Sidhartha~Dey,
Qirui Wu,
Pedro Hespanhol\\
Anil~Aswani,
Matthew~Johnson-Roberson,
and~Ram~Vasudevan
\thanks{This work was supported by a grant from Ford Motor Company via the Ford-UM Alliance under award N022977.}%
\thanks{M. Porter, A. Joshi, and R. Vasudevan are with the Department of Mechanical Engineering, University of Michigan, Ann Arbor, MI 48103 USA (e-mail:~matthepo@umich.edu;~arnavj@umich.edu;~ramv@umich.edu).}%
\thanks{S. Dey is with the Robotics Institute, University of Michigan, Ann Arbor, MI 48103 USA (e-mail:~siddey@umich.edu).}%
\thanks{Q. Wu is with the Department of Computer Engineering, University of Michigan, Ann Arbor, MI 48103 USA (e-mail:~wuqirui@umich.edu).}%
\thanks{A. Aswani and P. Hespanhol are with the Department of Industrial Engineering and Operations Research, University of California Berkeley, Berkeley, CA 94720 USA (e-mail:~aaswani@berkeley.edu;~pedrohespanhol@berkeley.edu).}%
\thanks{M. Johnson-Roberson is with the Department of Naval Architecture, University of Michigan, Ann Arbor, MI 48103 USA (e-mail:~mattjr@umich.edu).}%
}

\maketitle

\begin{abstract}
Networked robotic systems, such as connected vehicle platoons, can improve the safety and efficiency of transportation networks by allowing for high-speed coordination.
To enable such coordination, these systems rely on networked communications.
This can make them susceptible to cyber attacks. 
Though security methods such as encryption or specially designed network topologies can increase the difficulty of successfully executing such an attack, these techniques are unable to guarantee secure communication against an attacker. 
More troublingly, these security methods are unable to ensure that individual agents are able to detect attacks that alter the content of specific messages. 
To ensure resilient behavior under such attacks, this paper formulates a networked linear time-varying version of dynamic watermarking in which each agent generates and adds a private excitation to the input of its corresponding robotic subsystem. 
This paper demonstrates that such a method can enable each agent in a networked robotic system to detect cyber attacks. 
By altering measurements sent between vehicles, this paper illustrates that an attacker can create unstable behavior within a platoon.
By utilizing the dynamic watermarking method proposed in this paper, the attack is detected, allowing the vehicles in the platoon to gracefully degrade to a non-communicative control strategy that maintains safety across a variety of scenarios.
\end{abstract}
\section{Introduction}\label{sec:intro}
The adoption of intelligent transportation systems (ITS) promises to reduce road congestion while simultaneously improving road user safety.
Vehicle to vehicle (V2V) communications, in particular,  have been shown to improve efficiency and safety by enabling coordination \cite{Arem2006,englund2016}.
However, this reliance on outside communication introduces additional security concerns \cite{Gerla2015,Dominic2016,Amoozadeh2015}.
This paper considers the cooperative task of vehicle platooning and illustrates that a malicious agent can cause collisions by altering V2V communications within a platoon.
To address this challenge, we formulate a technique that uses a private excitation added to the input of each vehicle to identify whether an attack is taking place. 
The proposed method enables individual vehicles within a platoon to identify such attacks and gracefully degrade to a non-communicative control strategy that ensures safe, resilient behavior of the networked system.

\subsection{Vehicle Platooning}\label{sec:int_platooning}
Robotic vehicle platooning has been shown to decrease fuel consumption by reducing air drag \cite{bonnet2000fuel,Alam2010,liang2014fuel,mcauliffe2017fuel}, while also improving throughput on roads by reducing the occurrence of bottlenecks \cite{Chien1992Automatic,hall2005vehicle}. 
To achieve these performance improvements without creating phantom traffic jams or crashes \cite{Sugiyama2008,Flynn2009}, the longitudinal controller for these robotic platoons must be \emph{string stable} meaning that perturbations must be dampened by subsequent vehicles in the string \cite{herman1959traffic,chu1974decentralized,swaroop1996string,ploeg2013lp,feng2019string}.

For platoons that do not rely on V2V communication, string stability can be satisfied by maintaining a constant headway between vehicles in the platoon \cite{ioannou1993autonomous,swaroop1994comparision,zhou2004string}.
Unfortunately, since the headway is found by dividing the bumper-to-bumper following distance by the speed, this leads to
conservatively sized gaps between vehicles that grow as the speed increases.
Since the reduction in air drag becomes less prominent as the following distance increases, the methods that avoid utilizing V2V communication are unable to achieve all the potential energy efficiency benefits of robotic vehicle platooning. 

In contrast, for connected vehicles, the additional road user information can be used to adopt a constant spacing policy while still preserving string stability.
Even under limited V2V communication with just one or two neighboring vehicles \cite{stankovic2000decentralized,naus2010string,ploeg2013controller,swaroop1999constant,Seiler2004,Xiao2009}, energy efficiency and throughput can increase dramatically due to the reduction in following distance.
However, these communication channels introduce vulnerabilities to cyber attacks.

\subsection{Securing ITS}
The use of ad-hoc networks to facilitate collaborative actions, such as in impromptu platooning, results in additional security vulnerabilities.
Namely, the trustworthiness of communications received over the network is not guaranteed.
Therefore, to ensure the security of ITS, both the source of a communication and the communication's content must be verified.
The former has been widely studied as discussed in \citet{Hussain2020}, \citet{Bagga2020}, and references therein.
The latter has received a good deal of attention as well \cite{Al-Jarrah2019}.
However, methods for deriving content trustworthiness vary greatly. 
Some methods look for abnormalities in content without leveraging information on the underlying system \cite{Bezemskij2016,Kang2016,mehdi2017game}.
In contrast, \emph{model-aware} methods incorporate a system model in their trust analysis \cite{Bißmeyer2012,Hespanhol2018,Mo2012,Murguia2016CUSUMSensors,Umsonst2018}.
This allows for plausibility of the content of a communication to be evaluated in the context of the underlying system.

This paper focuses on one particular model-aware method, dynamic watermarking, in which each vehicle generates and adds a random excitation called a watermark to its inputs.
While each watermark is only known to the vehicle that generated it, the effect of each watermark propagates dynamically through the state of the system and into the measurements of surrounding vehicles.
Each vehicle observes the correlation of its watermark with the communicated measurements.
Then, using knowledge of the system dynamics, this correlation is compared with the expected correlation under normal operation to decide if the communication content is trustworthy.
In doing so, dynamic watermarking has been shown to detect a wider range of attack models than other methods while making few assumptions on the underlying system model \cite{weerakkody2016info,Weerakkody2017exposing,porter2019simulation}.

Originally, dynamic watermarking was derived for single input single output (SISO) linear time-invariant (LTI) systems \cite{Mo2009}, but extensions to 
multiple input multiple output (MIMO) LTI systems \cite{Satchidanandan2017,Hespanhol2017}, 
networked LTI systems \cite{Hespanhol2018}, 
switching LTI systems \cite{hespanhol2019sensor}, 
linear time-varying (LTV) systems \cite{porter2019ltvTAC,porter2020deception},
and simple non-linear systems have also been derived \cite{Ko2016,ko2019dynamic}.
Furthermore, detecting attacks on connected vehicle platoons using dynamic watermarking has also been previously studied \cite{Hespanhol2018}, \cite{ko2019dynamic}.
However, each case makes a strict assumption on the system model.
In particular, \citet{Hespanhol2018} assumes that each vehicle observes the entire platoon state and that the platoon maintains a constant velocity and road grade allowing it to be modeled as a networked LTI system. 
While the latter assumption is often satisfied in the current use case, the expansion of ITS will allow for impromptu platooning in less controlled environments where these assumptions are no longer valid.
More recently, \citet{ko2019dynamic} apply dynamic watermarking to a nonlinear platoon model.
This allows them to make guarantees of detection directly for the nonlinear system as opposed to a linearized approximation of the system.
However, their method does not readily extend beyond simple kinematic models, which do not account for non-negligible dynamic effects such as tire slip.
As described in \citet{kong2015kinematic}, to realistically model a vehicle's behavior one does need to include such dynamic tire slip terms.

In this paper, we propose a new method of dynamic watermarking developed around a networked LTV model. 
The proposed method requires fewer assumption than \citet{Hespanhol2017} and, unlike \citet{ko2019dynamic}, can be readily applied to platoons with non-negligible dynamic effects.

\subsection{Contributions}
The contributions of this paper are threefold.
First, we develop a notion of dynamic watermarking for networked LTV systems.
In particular, we derive statistical tests for detecting the presence of attacks. 
Second, we develop a mitigation strategy that allows a robotic vehicle platoon to gracefully degrade in the event of an attack. 
Third, networked LTV dynamic watermarking is applied to a simulated platoon of autonomous vehicles.
The simulation uses empirically found nonlinear dynamics to provide a realistic illustration of our proposed method.
Video of the simulated experiments can be found in \cite{Porter_vid_2020}.

The remainder of this paper is organized as follows.
Section \ref{sec:notation} overviews the notation used throughout the paper.
Section \ref{sec:platooning} outlines the problem of implementing attack detection for robotic vehicle platoons and illustrates how the system can be described using a networked LTV system.
Section \ref{sec:system} defines a general controller and observer structure and makes some assumptions on the system.
Section \ref{sec:tests} derives the statistical tests for distributed attack detection using the generalized networked LTV system.
Section \ref{sec:consider} provides practical methods for implementing the statistical tests on a real-world system.
Section \ref{sec:design} derives two robotic vehicle platoon controllers and observers that satisfy the conditions outlined in Section \ref{sec:system}.
In addition, a third controller and observer that does not use V2V communication is also derived to facilitate a graceful degradation of the platoon in the event of an attack.
Section \ref{sec:sim} illustrates the ability of the proposed algorithm to detect attacks using a simulated vehicle platoon.
We draw conclusions in Section \ref{sec:conclusion}.

\section{Notation}\label{sec:notation}
This section describes the notation used throughout the remainder of this paper.
We denote a variable $X$ for vehicle $i$ at step $n$ by $X_{i,n}$.
The zero matrix of size $i$ by $j$ is denoted $0_{i\times j}$.
The identity matrix of size $i$ is denoted $I_i$.
The block diagonal matrix with the blocks $M_1,~M_2,~M_3,\cdots$ on the diagonal is denoted diag$(M_1,~M_2,~M_3,\cdots)$.
The cardinality of a set $H$ is denoted $\text{card}(H)$.

The Wishart distribution with scale matrix $\Sigma$ and $i$ degrees of freedom is denoted $\mathcal{W}(\Sigma,i)$ \cite[Section 7.2]{anderson2003}.
The multivariate Gaussian distribution with mean $\mu$  and covariance $\Sigma$ is denoted $\mathcal{N}(\mu,\Sigma)$.
The matrix Gaussian distribution with mean $\mathcal{M}$, and parameters $\Sigma$ and $\Omega$ is denoted $\mathcal{N}(\mathcal{M},\Sigma,\Omega)$.
The expectation of a random variable $a$ is denoted $\mathds{E}[a]$.
\section{Robotic Vehicle
Platooning}\label{sec:platooning}
\begin{figure*}
    \centering
    \includegraphics[width=\textwidth]{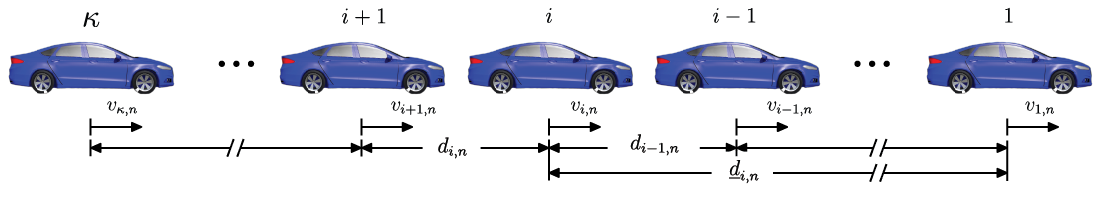}
    \caption{The platooning state at step $n$ is described by the velocity of each vehicle $v_{1,n},\hdots,v_{\kappa,n}$ and the distances between vehicles $d_{1,n},\hdots,d_{\kappa-1,n}$. The distance from a vehicle in an arbitrary position $i$ to the lead vehicle is $\underline{d}_{i,n}$. 
    }
    \label{fig:platoon}
\end{figure*}
This section illustrates the ability of an LTV system model to describe a robotic vehicle platoon.
The general approach is discussed and the resulting system is presented. A thorough derivation of the LTV system used in this work can be found in Appendix \ref{sec:derive}.
Whereas previous literature in platooning has often assumed a constant velocity, the use of an LTV system model can describe platooning with fewer assumptions.
Moreover, this system model provides a general structure for deriving networked LTV dynamic watermarking.

The task of vehicle platooning is often broken into two components, lane keeping and vehicle following \cite{Xu2018}.
The objective of lane keeping is to minimize the \emph{lateral error} i.e. the distance to the nearest point on the trajectory.
The task of vehicle following, as described in Section \ref{sec:int_platooning}, attempts to maintain a constant distance or a constant headway to the preceding vehicle in the platoon.
Note that the lead vehicle attempts to maintain a minimum headway when near other vehicles and a desired velocity otherwise.
In this work, we focus on the task of vehicle following since it requires V2V communication to reduce following distances between vehicles.
However, we also describe our lateral dynamics in Appendix \ref{sec:derive} for the sake of completeness.

We start with the empirically found non-linear model
\begin{align}
\begin{bmatrix} \dot{x}\\ \dot{y}\\ \dot{\psi}\\ \dot{v} \end{bmatrix}
&=
\begin{bmatrix} 
v\cos\psi - \dot{\psi}(c_8 + c_9v^2)\sin\psi \\ 
v\sin\psi + \dot{\psi}(c_8 + c_9v^2)\cos\psi \\ 
\frac{\tan(c_1\delta+c_2)v}{c_3 + c_4v^{2}} \\ 
c_5+c_6(v-v^{d}) + c_7(v-v^{d})^{2}
\end{bmatrix},
\label{eq:dynamics}\end{align}
where $(x,y)$ are the ground plane coordinates, $\psi$ is the vehicle heading, $v$ is the forward velocity, $\delta$ is the desired steering angle, $v^d$ is the desired velocity, and $c_1,\hdots,c_9$ are fitted constants which can be found in Table \ref{tab:dyn_const} in the appendix.
This model attempts to capture both the non-linear dynamics of the power-train and various dynamic effects such as tire slip.
Namely, the constants $c_1$ and $c_2$ are used to calibrate the steering angle, $c_4$ and $c_9$ are used to model the effect of tire slip on the angular and lateral velocities,  and $c_5$, $c_6$, and $c_7$ are used to approximate the drive train. 
Note that when $c_4=0$ and $c_9=0$ the equations for $\dot{x}$, $\dot{y}$, and $\dot{\psi}$ follow from \cite[Table 2.1]{rajamani2011vehicle} for a proper selection of $c_3$ and $c_4$.
The state of the longitudinal dynamics for a platoon of $\kappa$ vehicles, as illustrated in Figure \ref{fig:platoon}, consists of each vehicle's velocity $v_{1,n},\hdots,v_{\kappa,n}$, and the distances between subsequent vehicles $d_{1,n},\hdots,d_{\kappa-1,n}$.
For convenience we define
\begin{align}
    \underline{d}_{i,n}=\sum_{j=1}^{i-1}d_{i,n} \label{eq:distance_from_lead}
\end{align}
to denote the distance from the lead vehicle to vehicle $i$ in the platoon.
In this paper, the distance between vehicles is measured along the trajectory based on the center of each vehicle. 
This simplifies notation while still allowing for vehicle length to be accounted for in the platoons desired trajectory.

With the help of a few assumptions, we linearize and discretize these dynamics about a desired trajectory denoted $\{\bar{x}_{i,n},\bar{y}_{i,n},\bar{\psi}_{i,n},\bar{v}_{i,n},\bar{v}_{i,n}^d,\bar{\delta}_{i,n}\}_{n=0}^\infty$ for a step size of 0.05 seconds.
Following the linearization and discretization, the state vector and input vector take the form $x_n=[\Delta d_{1,n}~\cdots~\Delta d_{\kappa-1,n}~\Delta v_{1,n}~\cdots~\Delta v_{\kappa,n}]^\T$ and $u_{i,n}=\Delta v_{i,n}^d$ respectively
where $\Delta d_{i,n}=d_{i,n}-\bar{d}_{i,n}$, $\Delta v_{i,n} = v_{i,n}-\bar{v}_{i,n}$, and $\Delta v^d_{i,n}= v^d_{i,n}-\bar{v}^d_{i,n}$.
The resulting LTV system then satisfies
\begin{align}
    x_{n+1}&=A_nx_n+\sum_{i=1}^\kappa B_{i,n}u_{i,n}+w_n,\label{eq:state_update}
\end{align}
where $w_n$ is the process noise and matrices $A_{n}$ and $B_{i,n}$ are defined in Appendix \ref{sec:derive}.

Since the vehicle following task only seeks to maintain a desired velocity and spacing policy,  any longitudinal error with respect to a predefined trajectory will not be corrected by the controller.
As a result, the position of the platoon can drift along the trajectory.
Though this drift does not affect the derivations in this section, it does pose a practical problem which we resolve in Section \ref{sec:consider}.

We assume each vehicle is able to measure its own location, velocity, and the distance to the previous vehicle in the platoon. 
Using the location measurements, the lateral error and heading error are calculated for the lane keeping task.
For the vehicle following task the measurements for vehicle $i$ at step $n$ satisfy
\begin{align}
    y_{i,n}&=C_{i}x_n+z_{i,n},~i\in\{1,\hdots,\kappa\},
    \label{eq:measurement_model}
\end{align}
where $z_{i,n}$ is the measurement noise and $C_{i,n}$ takes one of two forms depending on the choice of controller.
In the first case, the leader measures its own velocity, while each of the other vehicles measures both their own velocity and the distance to the preceding vehicle.
In the second case, we assume that the first vehicle also communicates its location along the trajectory which can be used by other vehicles to calculate the distance to the lead vehicle.
However, we model this as the first vehicle also measuring the distance to all following vehicles. Both versions of the matrix $C_{i,n}$ can be found in Appendix \ref{sec:derive}.
\section{Networked LTV System}\label{sec:system}
This section provides necessary assumptions on the system noise, defines a generic controller and observer structure, and derives the closed loop dynamics.
This added structure allows us to formulate our proposed detection algorithm.

Consider an LTV system of $\kappa$ vehicles satisfying \eqref{eq:state_update} and \eqref{eq:measurement_model}, where $x_n, w_n \in \mathbb{R}^{p}$, $u_{i,n}\in\mathbb{R}^{q_i}$, $y_{i,n},z_{i,n}\in\mathbb{R}^{m_i}$.
We then make the following assumption.
\begin{asmp}
The process and measurement noise are mutually independent and Gaussian distributed such that
\begin{align}
    w_n&\sim\mathcal{N}(0_{p\times1},\Sigma_{w,n}),
\label{eq:process_noise}\\
    z_{i,n}&\sim\mathcal{N}(0_{m_i\times1},\Sigma_{z_i,n}).
\label{eq:measurement_noise}\end{align}
\end{asmp}
\noindent While the linearization of non-linear systems, generally does not result in Gaussian distributed noise, this assumption allows us to derive our proposed method using a statistical basis.
Furthermore, this assumption has not hindered the efficacy of LTV dynamic watermarking in non-networked settings \cite{porter2020deception}.

The communication from vehicle $j$ to vehicle $i$ at step $n$ is denoted $s_{(i,j),n}$ and takes the form 
\begin{align}
    s_{(i,j),n}=y_{j,n}+a_{(i,j),n},\label{eq:com_measurement}
\end{align}
where $a_{(i,j),n}$ is an additive attack on the communication channel.
In the platoon, this style of attack can take the form of a man in the middle, where the attacker intercepts and alters measurements being communicated between vehicles, or a malicious vehicle that sends out false measurements. 
Since vehicle $i$ already has its own measurement, the additive attack $a_{(i,i),n}$ is zero for $i\in\{1,\hdots,\kappa\}$.
To reduce computational overhead, particularly in larger platoons, some V2V communications may not be used.
We define the following set to ease later notation.
\begin{align}
    H\subseteq\left(\{1,\hdots,\kappa\}\times\{1,\hdots,\kappa\}\right),
\end{align}
where the communication from vehicle $j$ to vehicle $i$ is \emph{active} if $(i,j)\in H$. 
Moreover, we define the set
\begin{align}
    H_i=\{j|(i,j)\in H\},\label{eq:partial_H}
\end{align}
which contains the indices of vehicles that send measurements to vehicle $i$.

The input from vehicle $i$ comes in the form
\begin{align} \label{eq:control_input}
    u_{i,n}=K_{i,n}\hat{x}_{i,n}+e_{i,n},
\end{align}
where $K_{i,n}$ is a user defined control gain and $e_{i,n}\in\mathbb{R}^{q_i}$ is the watermark generated by vehicle $i$. 
The watermark is Gaussian distributed such that $e_{i,n}\sim\mathcal{N}(0_{q_i\times 1},\Sigma_{e_i})$ and is mutually independent with $z_{i,n},~w_n$. 
Vehicle $i$ observes a subset of the platoons state denoted $\hat{x}_{i,n}$ using the linear functional observer
\begin{fleqn}
\begin{equation}
    \hat{x}_{i,n+1}=M_{i,n}\hat{x}_{i,n}+N_{i}B_{i,n}e_{i,n}-\sum_{j\in H_i} L_{(i,j),n}s_{(i,j),n},\ltxlabel{eq:observer_update}
\end{equation}
\end{fleqn}
where the gain $L_{(i,j),n}$ and update matrix $M_i$ are user defined.
This observer follows a discrete time version of the observer first introduced in \citet{luenberger}.
However, since the watermark of other vehicles is unknown to vehicle $i$ their input is not included in \eqref{eq:observer_update}.

Next we derive the closed loop dynamics.
For ease of notation, we combine each vehicle's measurement noise and watermark such that
 \begin{align}
    z_{n}^\T&=\begin{bmatrix}z_{1,n}^\T & \cdots & z_{\kappa,n}^\T\end{bmatrix},\\
    \Sigma_{z,n}&=\text{diag}(\Sigma_{z_1,n},\hdots,\Sigma_{z_\kappa,n}),\\
    e_{n}^\T&=\begin{bmatrix}e_{1,n}^\T & \cdots & e_{\kappa,n}^\T\end{bmatrix},\\
    \Sigma_e&=\text{diag}(\Sigma_{e_1},\hdots,\Sigma_{e_\kappa}).
\end{align}
The additive attacks are also combined such that
\begin{align}
    a_{i,n}^\T&=\begin{bmatrix}a_{(i,1),n}^\T & \cdots & a_{(i,\kappa),n}^\T\end{bmatrix},\\
    a_{n}^\T&=\begin{bmatrix}a_{1,n}^\T & \cdots & a_{\kappa,n}^\T\end{bmatrix}.
\end{align}
Then we can write the closed loop system as 
\begin{align}
    \bar{x}_{n+1}&=
    \bar{A}_n
    \bar{x}_n+
    \bar{B}_ne_n+
    \begin{bmatrix}w_n\\0\\\vdots\\0
    \end{bmatrix}-\bar{L}_n
    \begin{bmatrix}0\\
    z_n+a_{1,n}\\\vdots\\z_n+a_{\kappa,n}\end{bmatrix}\label{eq:closed_loop},
\end{align}
where
$\bar{x}_n^\T=\begin{bmatrix}x_{n}^\T&\hat{x}_{1,n}^\T&\cdots&\hat{x}_{\kappa,n}^\T\end{bmatrix}$, and $\bar{A}_n,~\bar{B}_n,$ and $\bar{L}_n$ are found in Appendix \ref{sec:cl_mat}.
On the closed loop system we make the following assumption.
\begin{asmp}
Consider a closed loop system satisfying \eqref{eq:closed_loop}. If $a_{(i,j),n}=0_{m_i\times1}$ for all $(i,j)\in H$ and $n\in\mathbb{N}$, there exists positive constants $\eta_1,\eta_2\in\mathbb{R}$ such that
\begin{align}
    \|\mathds{E}[\bar{x}_n\bar{x}_n^\T]\|<\eta_1,\label{eq:bound1}\\
    \|\left(\mathds{E}[\bar{x}_n\bar{x}_n^\T]^{-1}\right)\|<\eta_2.\label{eq:bound2}
\end{align}
\end{asmp}
\noindent This assumption ensures that, despite being time varying, the covariance of the closed loop system and its inverse are well defined for all time.
For the system in this paper, \eqref{eq:bound1} holds since the controllers and observers we later derive render the system stable. 
Furthermore, we note that our system also satisfies \eqref{eq:bound2} as a result of the system noise propagating through the state vector.
\section{Networked LTV Tests}\label{sec:tests}
This section derives the statistical tests for networked LTV dynamic watermarking.
These tests utilize the difference between the observed state and the measured state called the \emph{measurement residual}.
The measurement residual for each $(i,j)\in H$ is formally defined as
\begin{align}\label{eq:residual}
    r_{(i,j),n}=U_{(i,j)}\hat{x}_{i,n}-W_{(i,j)}s_{(i,j),n},
\end{align}
where $r_{(i,j),n}\in\mathbb{R}^{p_{(i,j)}}$.
Furthermore, the matrices $U_{(i,j)}$ and $W_{(i,j),n}$ are defined such that
\begin{align}
    U_{(i,j)}N_i=W_{(i,j)}C_{j} \quad\forall (i,j)\in H.
\end{align}
Here, the matrices $U$ and $W$ are used to select the common state information between the observed state for vehicle $i$ and measurement from vehicle $j$.

The proposed detection scheme focuses on the covariance of the measurement residual $r_{(i,j),n}$ and its correlation with the watermark $e_{i,n-\rho_{(i,j)}-1}$ where $\rho_{(i,j)}$ is a user defined delay.
In the case that the delayed watermark is correlated with the common state information of the un-attacked measurement, 
an attacker cannot scale the true measurement signal without altering this correlation.
As a result, monitoring the watermarks correlation with the measurement residual (which is a function of the measurement) allows our proposed algorithm to detect attacks that scale or remove the true measurement.
To ensure a correlation between the delayed watermark and the common state information of the measurement, $\rho_{(i,j)}$ is selected to account for the time needed for the watermark to propagate through the system and into the measurement $s_{(i,j),n}$. 
For example, suppose vehicle $i$ adds $e_{i,n}$ to its desired velocity input $v_{i,n}^d$ at step $n$. 
The effect of $e_{i,n}$ will be seen in vehicle $i$'s velocity measurement and vehicle $i+1$'s distance measurement at step $n+1$.
However, vehicle $i+2$'s measurement will require additional time for the effect of $e_{i,n}$ to be seen in its measurement.
Though in general there is no guarantee that a given vehicles watermark will eventually propagate through the entire system, by our choice of controllers and corresponding communication structures, we guarantee the existence of a $\rho_{(i,j)}$ that ensures a correlation between $e_{i,n-\rho_{(i,j)}-1}$ and $W_{(i,j)}s_{(i,j),n}$ for each $(i,j)\in H$.
Moreover, given a particular $\rho_{(i,j)}$ we provide a sufficient condition for a non-zero correlation in Appendix \ref{sec:corr_prop}.

To monitor both the covariance of the measurement residual and its correlation with the watermark, we aim to observe the normalized outer product of the vector 
$[e_{i,n-\rho_{(i,j)}-1}^\T~r_{(i,j),n}^\T]^\T$,
which follows a Gaussian distribution when the platoon is un-attacked such that  
\begin{align}
    \begin{bmatrix}
    e_{i,n-\rho_{(i,j)}-1}\\
    r_{(i,j),n}
    \end{bmatrix}
    \sim\mathcal{N}\left(0_{(p_{(i,j)}+q_i)\times1},\Sigma_{(i,j),n}\right)\label{eq:combined_res}.
\end{align}
To this end, we start by defining a matrix normalizing factor similar to that of \citet{porter2019ltvTAC}
\begin{align}
    V_{(i,j),n}&=\Sigma_{(i,j),n}^{-\frac{1}{2}},
\end{align}
to create a new vector with constant covariance
\begin{align}
    \bar{r}_{(i,j),n}&=V_{(i,j),n}\begin{bmatrix}e_{i,n-\rho_{i,j}-1}\\r_{(i,j),n}\end{bmatrix}\sim\mathcal{N}(0_{(p_{(i,j)}+q_i)\times1},I).\label{eq:normalized_concat_res}
\end{align}
Next, we form the matrix
\begin{align}
    P_{(i,j),n}&=\begin{bmatrix}\bar{r}_{(i,j),n-\ell_{(i,j)}+1} & \cdots & \bar{r}_{(i,j),n}\end{bmatrix},\label{eq:res_window}
\end{align}
which, under the assumption of no attack, is distributed as
\begin{align}
    P_{(i,j),n}&\sim\mathcal{N}\left(0_{(p_{(i,j)}+q_i)\times\ell_{(i,j)}},I_{p_{(i,j)}+q_i},G_{(i,j),n}\right),
\end{align}
where
\begin{align}
    G_{(i,j),n}=\frac{\mathds{E}[P_{(i,j),n}^\T P_{(i,j),n}]}{p_{(i,j)}+q_i}.
\end{align}
Note that the values of $V_{(i,j),n}$ and $G_{(i,j),n}$ can be derived as functions of the dynamics and the covariances of the watermarks, process noise, and measurement noise.
However, in real word applications these covariances must be approximated.
To avoid compounding error, we derive methods for approximating both $V_{(i,j),n}$ and $G_{(i,j),n}$ in Section \ref{sec:consider}.
Finally, we form the matrix
\begin{align}
    S_{(i,j),n}=P_{(i,j),n}G_{(i,j),n}^{-1}P_{(i,j),n}^\T,\label{eq:S_mat}
\end{align}
which is distributed according to a Wishart distribution with scale matrix $I_{p_{(i,j)}+q_i}$ i.e.
\begin{align}
    S_{(i,j),n}\sim\mathcal{W}(\ell_{(i,j)},I_{p_{(i,j)}+q_i}).
\end{align}
Using this distribution, a statistical test can be implemented using the negative log likelihood of the scale matrix given the observation $S_{(i,j),n}$
\begin{align}
    \mathcal{L}_{(i,j)}&=(1-\ell_{(i,j)}+p_{(i,j)}+q_i)\log(|S_{(i,j),n}|)+\nonumber\\
    &\quad+\text{trace}(S_{(i,j),n}).\label{eq:likelihood}
\end{align}
If $\mathcal{L}_{(i,j)}$ exceeds a user defined threshold, vehicle $i$ signals that an attack has likely occurred.
The threshold can be set to satisfy a desired performance metric such as the false alarm rate.

\SetKwFor{Loop}{Loop}{}{EndLoop}
\begin{algorithm}[h]
 \SetAlgoLined
 set cntrl\_lvl=3 (resp. 2)\;
 set $\hat{x}_{i,0}=0_{p_i}$\;
 set $n=0$\;
 set $H_i$ using \eqref{eq:fully_H} (resp. \eqref{eq:leader_H}) and \eqref{eq:partial_H}\;
 \Loop{}{
    \For{$j\in H_i$}{
        get $s_{(i,j),n}$\;
        set detect$_{(i,j),n}=0$\;
        \If{$n>\rho_{(i,j)}$ \textup{and cntrl\_lvl$\neq$1}}{
            set $V_{(i,j),n} = \bar{V}_{(i,j),h_i(n)}$\;
            set $\bar{r}_{(i,j),n}$ using \eqref{eq:normalized_concat_res}\;
        }
        \If{$n>\rho_{(i,j)}+\ell_{(i,j)}$}{
            set $G_{(i,j),n} = \bar{G}_{(i,j),h_i(n)}$\;
            set $\mathcal{L}_{(i,j),n}$ using \eqref{eq:res_window}, \eqref{eq:S_mat}, and \eqref{eq:likelihood}\;
            \If{$\mathcal{L}_{(i,j),n}>\textup{Threshold}_{(i,j)}$}{
                set $\text{detect}_{(i,j),n}$=1\;
            }
            \If{$(\sum_{k=n-39}^n \textup{detect}_{(i,j),n})>24$}{
                set cntrl\_lvl=1\;
                reformat $\hat{x}_{i,n}$
            }
        }
    }
    find $h_i(n)$ from high res. trajectory\;
    set $\{\bar{v}_{j,n},\bar{v}_{j,n}^d\}=\{\underline{v}_{j,h_i(n)},\underline{v}_{j,h_i(n)}^d\}$\;
    \uIf{\textup{cntrl\_lvl=3 (resp. 2)}}{
    set $M_{i,n},K_{i,n}$ using \eqref{eq:fully_M},\eqref{eq:fully_K} (resp. \eqref{eq:leader_M},\eqref{eq:leader_K})\;
    \For{$j\in H_i$}{
    set $L_{(i,j),n}$ using \eqref{eq:fully_L} (resp. \eqref{eq:leader_L1}-\eqref{eq:leader_L4})\;
    }
    sample $e_{i,n}$\;
    set $u_{i,n}=K_{i,n}\hat{x}_{i,n}+e_{i,n}+\bar{u}_{i,n}$\;
    }
    \Else{
    set $M_{i,n},~L_{(i,i),n},~K_{i,n}$ using \eqref{eq:acc_M},\eqref{eq:acc_L},\eqref{eq:acc_K}\;
    set $u_{i,n}=K_{i,n}\hat{x}_{i,n}+\bar{u}_{i,n}$\;
    }
    set $\hat{x}_{i,n+1}$ using \eqref{eq:observer_update}\;
    send $u_{i,n}$\;
    set $n=n+1$\;
}
\caption{Longitudinal Control for Vehicle $i$}\label{alg:cntrl}
\end{algorithm}
\section{Considerations for Implementation}\label{sec:consider}
This section derives practical solutions to the issue of drift along the trajectory as described in Section \ref{sec:platooning}.
More specifically, we devise a method for allowing the linearization to drift along the trajectory by constructing the discretized trajectory in a receding horizon fashion. 
Furthermore, we derive an approximation scheme for the normalization matrices used in the statistical tests that accommodates this drift as well.

We start by creating a \emph{high-resolution} trajectory denoted  $\{\underline{x}_{i,k},\underline{y}_{i,k},\underline{\psi}_{i,k},\underline{v}_{i,k},\underline{v}_{i,k}^d,\underline{\delta}_{i,k}\}_{k=0}^\infty$ which uses a step size 0.005 s or $\frac{1}{10}$ of the step size used when discretizing the dynamics, controllers, and observers.
At the $n^\text{th}$ 0.05 s step of the system, vehicle $i$ finds its closest point on the high-resolution trajectory denoted $h_i(n)$, then sets
\begin{align}
    &\{\bar{x}_{j,n},\bar{y}_{j,n},\bar{\psi}_{j,n},\bar{v}_{j,n},\bar{v}_{j,n}^d,\bar{\delta}_{j,n}\}=\nonumber\\
    &\quad\{\underline{x}_{j,h_i(n)},\underline{y}_{j,h_i(n)},\underline{\psi}_{j,h_i(n)},\underline{v}_{j,h_i(n)},\underline{v}_{j,h_i(n)}^d,\underline{\delta}_{j,h_i(n)}\}
\end{align}
for each $j\in\{1,\hdots,\kappa\}$.
Note that the step in the high-resolution trajectory found by each vehicle may differ due to error in the following distances between vehicles.
However, this difference remains small since the longitudinal controller regulates the following distance error.

Similarly to the discretized trajectory, the matrix normalizing factor $V_{(i,j),n}$ and the auto-correlation normalizing factor $G_{(i,j),n}$ are selected in a receding horizon fashion.
We accomplish this by approximating both normalizing factors for each step in the high resolution trajectory then selecting the appropriate normalizing factors at each 0.05 s step of the system using the index of the high resolution trajectory $h_i(n)$.
To this end, we denote the sample covariance at index $k$ of the high-resolution trajectory using the ensemble average
\begin{align}
    \bar{\Sigma}_{(i,j),k}= \frac{1}{f_{i,k}}\sum_{\tau=1}^{\tau^*}\sum_{h_i^{(\tau)}(n)=k}
    \begin{bmatrix}
    e_{i,n-\rho_{(i,j)}-1}^{(\tau)}\\r_{(i,j),n}^{(\tau)}
    \end{bmatrix}
    \begin{bmatrix}
    e_{i,n-\rho_{(i,j)}-1}^{(\tau)}\\r_{(i,j),n}^{(\tau)}
    \end{bmatrix}^\T\label{eq:sample_cov_estimate}
\end{align}
where the superscript $(\tau)$ denotes the realization number, $\tau^*$ is the total number of realizations, and $f_{i,k}=\text{card}(\{(\tau,n)~|~\tau\in\{1,\dotsc,\tau^*\}, h_i^{(\tau)}(n)=k\})$.
In the event that no samples are available for a given $k$, we set  $\bar{\Sigma}_{(i,j),k}=0_{p_{(i,j)+q_i}}$.
Since we are limited to having a finite number of realizations, there is no guarantee that the sample covariance matrices all have a sufficient number of samples to be invertable.
To overcome this obstacle, we take a weighted average such that
\begin{align}
    \bar{V}_{(i,j),k}=\left(\frac{1}{b_{i,k}}
    \sum_{\epsilon=k-10}^{k+10}\sigma^{|k-\epsilon|}\bar{\Sigma}_{(i,j),k}\right)^{-\frac{1}{2}},
\end{align}
where $0<\sigma<1$ is used to reduce the weight of samples that are farther away (we use $\sigma=0.8$) and $b_k$ is the sum of the weights for which samples exist
\begin{align}
    b_{i,k}=\sum_{\substack{\epsilon=k-10,\hdots k+10\\f_{i,\epsilon}>0}}\sigma^{|k-\epsilon|}.
\end{align}
Here, the range of indices in the summation was chosen such that the number of realizations needed to ensure invertability of $\bar{V}_{(i,j),n}$ should be no more than the dimension of the sample covariance matrix.
However, the number of samples should exceed this value to ensure a good approximation.
Finally, the matrix normalizing factor is approximated at each step as
\begin{align}
    V_{(i,j),n}\approx\bar{V}_{(i,j),h_i(n)}.\label{eq:matrix_norm_factor}
\end{align}

The auto-correlation normalizing matrices $G_{(i,j),n}$ can then be approximated using  \eqref{eq:normalized_concat_res}-\eqref{eq:res_window} and the approximate matrix normalization factor using the ensemble average
\begin{align}
    &\bar{G}_{(i,j),k}=
    \frac{1}{(p_{(i,j)}+q_i)g_{i,k}}\times\nonumber\\
    &\qquad\times\sum_{\tau=1}^{\tau^*}
    \sum_{\substack{h_i^{(\tau)}(n)\leq k\\h_i^{(\tau)}(n+1)>k}}\Bigg(
    \frac{1}{|h_i^{(\tau)}(n+1)-h_i^{(\tau)}(n)|}\times\nonumber\\
    &\qquad\times\left(
    |k-h_i^{(\tau)}(n+1)|
    \left(P_{(i,j),n}^{(\tau)}\right)^\T 
    P_{(i,j),n}^{(\tau)}+\right.
    \nonumber\\
    &\qquad+\left.|k-h_i^{(\tau)}(n)|
    \left(P_{(i,j),n+1}^{(\tau)}\right)^\T 
    P_{(i,j),n+1}^{(\tau)}\right)\Bigg),\label{eq:auto_corr_ensemble}
\end{align}
where the superscript $(\tau)$ denotes the realization number, $\tau^*$ is the total number of realizations, and $g_{i,k}=\text{card}(\{(\tau,n)|\tau\in\{1,\dotsc,\tau^*\},~h_i^{(\tau)}(n)\leq k,~h_i^{(\tau)}(n+1)> k\})$.
The auto-correlation normalizing matrix is then approximated at each step as
\begin{align}\label{eq:auto_corr}
G_{(i,j),n}\approx \bar{G}_{(i,j),h_i(n)}.
\end{align}

\section{Controller Design}\label{sec:design}

This section details the 3 different longitudinal controllers utilized in the networked robotic system.
In each case, we provide a general control scheme and the parameters chosen in this work based on their performance in the simulated platoon.
The level 3 and level 2 controllers both attempt to maintain constant spacing between vehicles.
However, they utilize different communication strategies that allow us to compare the benefits of full communication between vehicles in the platoon and a more limited communication strategy.
The level 1 controller attempts to maintain constant headway without the aid of V2V communication.
As a result, level 1 control strategy is used after an attack is detected, allowing the platoon to gracefully degrade.
Note that the lateral controller has been included in Appendix \ref{sec:lat_control} for completeness.

We assume that all vehicles in the platoon use the same communication level and switch to the mitigation strategy simultaneously when an attack is detected. 
For each level, the longitudinal controller and observer follows the LTV form \eqref{eq:control_input}-\eqref{eq:observer_update} however the matrices $K_{i,n},~L_{(i,j),n},~N_i,~M_i,~U_{(i,j)},$ and $W_{(i,j)}$ are level dependent.
Due to their size, the observer matrices $N_i$, $L_{(i,j),n}$, and  $M_i$ are located in Appendix \ref{sec:obs_mat}. 

\textbf{Level 3:} To take advantage of full V2V communication between the connected vehicles, a "fully-connected" controller and observer are devised.
In this case
\begin{align}
    H=\{1,\dotsc,\kappa\}\times\{1,\dotsc,\kappa\}\label{eq:fully_H}
\end{align}
and the measurements follow the model in \eqref{eq:measurement_model},\eqref{eq:C1}.
The controller gains satisfy
\begin{align} 
    &K_{i,n} =\nonumber\\
    &\Big[\hspace*{-4pt}
        \begin{array}{c;{2pt/2pt}c;{2pt/2pt}}
            \frac{0.5}{2^{(i-1)}} \cdots \frac{0.5}{2^{(1)}} & 
            \frac{-0.5}{2^{(0)}} \cdots \frac{-0.5}{2^{(\kappa-i-1)}}
            \end{array}\nonumber\\
            &\qquad\begin{array}{c;{2pt/2pt}c}
            \frac{0.1}{2^{(i-1)}} \cdots \frac{0.1}{2^{(1)}} & 
            \frac{-0.1}{2^{(0)}} \cdots \frac{-0.1}{2^{(\kappa-i)}}
        \end{array}\hspace*{-4pt}\Big]\label{eq:fully_K}
\end{align}
for $i\in\{1,\hdots,\kappa\}$.
The idea here is to have each vehicle react to the velocity and following distance error of all other vehicles.
In doing so, the watermark of each vehicle affects all others in the platoon.
Furthermore, the magnitude of the control gains decays exponentially for vehicles further away in the platoon to reduce the combined effect of the watermarks especially in larger platoons.

As the platoon size increases, the number of communication channels and corresponding incoming messages for each vehicle increases.
This problem along with other physical limiting factors such as latency between vehicles at the ends of the platoon can prove troublesome in larger, fully-connected platoons.
Along with these limitations, the fact that the visibility of the watermark from vehicle $i$ in vehicle $j$ reduces as the platoon positions $i$ and $j$ are further apart means there are diminishing returns, from an attack detection standpoint, to having level 3 communication in larger platoons.

\textbf{Level 2:} 
In light of the limitations associated with level 3, a less connected strategy where 
\begin{align}
    H=\{(i,j)\in\{1,\dotsc,\kappa\}^2~|~j=1\text{ or }|i-j|\leq1\}\label{eq:leader_H}
\end{align}
and measurements that follow the measurement model defined in \eqref{eq:measurement_model}, \eqref{eq:C2} 
are considered for the level 2 communication.
In this case, each vehicle observes a subset of the states.
Namely, its distance to the lead vehicle, preceding vehicle, and following vehicle, and the velocities of each of these vehicles.

Next we derive the level 2 control strategy using inspiration from  \citet[Section 3.4]{swaroop1999constant}, in which the controller uses the state of the lead and preceding vehicles to calculate a desired acceleration as
\begin{align}\label{eq:leader_follow}
    \dot{v}_{i,n} &= \frac{1}{1+\gamma_{1}}[\dot{v}_{i-1,n} + \gamma_{1}\dot{v}_{1,n} - (\gamma_{2} + 0.6)(v_{i,n} - v_{i-1,n}) + \nonumber\\ &- 0.6\gamma_{2}(d_{i-1,n}-\bar{d}_{i-1,n}) - (\gamma_{3} + 0.6\gamma_{1})(v_{i,n} - v_{1,n}) +  \nonumber\\ 
    &+ 0.6\gamma_{3}(\underline{d}_{i,n}-\bar{\underline{d}}_{i,n})],
\end{align}
where $\gamma_1$ is used to shift the relative gain from the acceleration of the preceding vehicle to that of the leading vehicle, $\gamma_2$ adjusts the control gains corresponding to following distance and relative velocity to the previous vehicle, and $\gamma_3$ adjusts the control gains corresponding to the following distance and relative velocity to the lead vehicle.
To enact this control policy we start by setting $\gamma_1=0.2,~\gamma_2=1,$ and $\gamma_3=1.2$ to achieve a spacing error attenuation rate of less than 0.5 \cite[Equation 3]{swaroop1999constant}.
Then, since our controller will specify the desired velocity $v^d$ instead of the desired acceleration, we relate the two using the partial derivative 
\begin{align}
    \frac{\partial\dot{v}_{i,n}}{\partial v^d_{i,n}}
    =-1.2(c_6+2c_7(v_{i,n}-v_{i,n}^d)).
\end{align}
Further, we assume that the deviation in the acceleration of the lead vehicle and preceding vehicle are negligible allowing us to ignore the corresponding terms.
The resulting controller gain matrices satisfy
\begin{align}
    &K_{i,n}=\frac{1}{-1.2(c_6+2c_7(\bar{v}_{i,n}-\bar{v}^d_{i,n}))}\times\nonumber\\
    &\times\begin{cases}
    \begin{bmatrix}
    0 & -2.92 & 0
    \end{bmatrix} & i=1\\\noalign{\vskip3pt}
    \begin{bmatrix}
    1.32 & 0 & 2.92 & -2.92 & 0
    \end{bmatrix} & i=2\\\noalign{\vskip3pt}
    \begin{bmatrix}
    0.72 & 0.6 & 1.32 & 1.6 &-2.92
    \end{bmatrix} & i=\kappa\\\noalign{\vskip3pt}
    \begin{bmatrix}
    0.72 & 0.6 & 0 & 1.32 & 1.6 &-2.92 & 0
    \end{bmatrix} & \text{o/w},
    \end{cases}\label{eq:leader_K}
\end{align}
with the corresponding observed state $\hat{x}_{i,n}$ approximating $N_ix_{n}$ as defined in Table \ref{tab:obs_table1} located in Appendix \ref{sec:obs_mat}.

\textbf{Level 1:} In the event of an attack detection, communication between agents in the network should be severed to mitigate potential harm to the system.
To maintain operation of the platoon, the vehicles are able to switch to a non-communicative platoon strategy such that 
\begin{align}
    H=\{(i,i)~|~i\in\{1,\dotsc,\kappa\}\}.
\end{align}
Here each vehicle still measures and observes its own velocity and the distance to the preceding vehicle.

The non-communicative level 1 controller is inspired by the University of Michigan Transportation Research Institute's  (UMTRI)'s algorithm  for adaptive cruise control (ACC) \cite[Equation 1]{fancher2002evaluating} which satisfies
\begin{align}
    v_{i,n}^d&=
    v_{i-1,n} + \phi_1(f_{i-1,n} -T_h v_{i,n}) + \nonumber\\
    &\quad+ \phi_2(v_{i-1,n}-v_{i,n}), \label{eq:cruise_controller1}
\end{align}
where $\phi_1$ is the control gain for the error in the headway, $f_{i-1,n}$ is the bumper-to-bumper distance to the previous vehicle, and $\phi_2$ is the control gain on the derivative of the following distance.
In this work, we set $\phi_1=1,~\phi_2=0.2,$ and $T_h=1$.
Since \eqref{eq:cruise_controller1} is already linear the control gain $K_{i,n}$ is written directly as
\begin{align}
    K_{i,n}=\begin{cases}
    \begin{bmatrix}
    -1
    \end{bmatrix} & i=1\\\noalign{\vskip3pt}
    \begin{bmatrix}
    1 & 1.2 & -1.2
    \end{bmatrix} & i\neq 1,
    \end{cases}\label{eq:acc_K}
\end{align}
with the corresponding observed state $\hat{x}_{i,n}$ approximating $N_ix_{n}$ as defined in Table \ref{tab:obs_table1} located in Appendix \ref{sec:obs_mat}.
Note that a proportional gain is applied to the error in the lead vehicle's velocity since there is no preceding vehicle.

Since the level 1 control strategy is considered to be safe from cyber attacks across communication channels, we do not employ attack detection or a watermark after the switch is made.
Though the level 1 controller is less susceptible to cyber attacks, it is unable to maintain constant following distances without sacrificing safety.
In contrast, both the level 3 and level 2 control strategies enable the use of constant following distances.
As a result, these strategies can lead to significant improvement in fuel economy and throughput on roads.
\section{Simulated Experiments}\label{sec:sim}
\begin{figure}[t]
    \centering
    \includegraphics[width=0.48\textwidth]{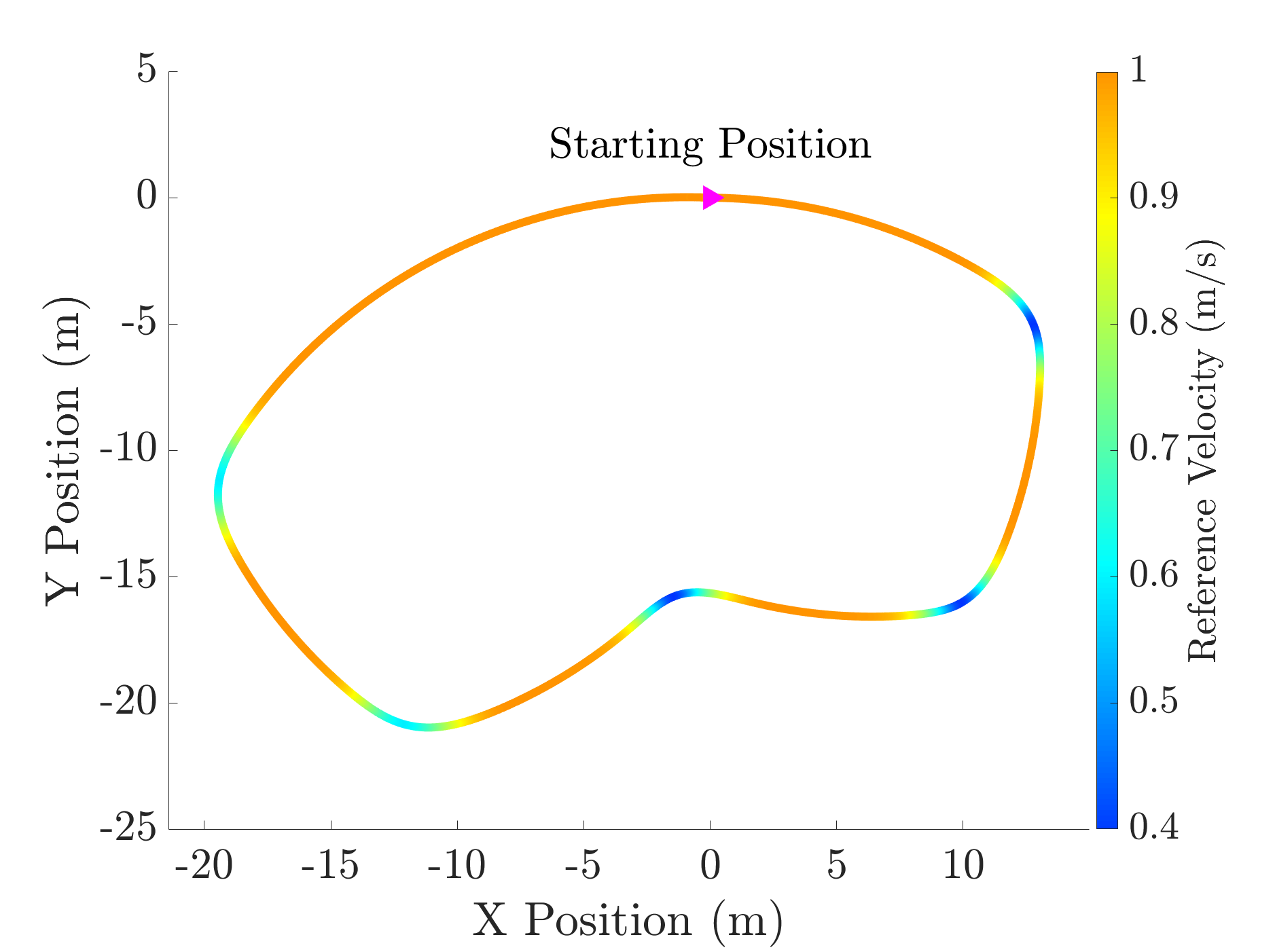}
    \caption{Reference trajectory of the lead vehicle in simulated platoon experiments. Each simulation consists of three laps.}
    \label{fig:velocitymap}
\end{figure}
This section illustrates the effectiveness of networked LTV dynamic watermarking on detecting attacks on V2V communication of platooning vehicles.
The experiment is implemented on a simulated platoon of four autonomous vehicles 
traveling three times around the looped path illustrated in
Figure \ref{fig:velocitymap}.
In the simulation, the vehicles have a vehicle length of 0.5m and drive according to the non-linear dynamics defined in \eqref{eq:dynamics}.

For simulations of the level 3 and 2 controllers, the platoon was tasked with maintaining a 1m constant following distance which was chosen based on the ability of each controller to maintain the desired distance as shown in Table \ref{tab:effect_of_watermark}.
Simulations of the level 1 controller had the platoon maintain a constant headway of 1 second.
 
Process and measurement noise as defined in \eqref{eq:process_noise} and \eqref{eq:measurement_noise} were also added at each step $n$, where $\Sigma_{w,n} = 1 \times 10^{-6} I_p$, and $\Sigma_{z_i,n} = \text{diag}(1 \times 10^{-4}, 1 \times 10^{-3}) \text{ for } i \in \{2, \dots, \kappa\}$. For the lead vehicle, the measurement noise covariance is $\Sigma_{z_{1},n} = 1 \times 10^{-3}$.
However, for the special case where the distance from each vehicle to the lead vehicle $\underline{d}_{i,n}$ is treated as a measurement the value $\Sigma_{z_{1},n} = \text{diag}(1 \times 10^{-4}I_{\kappa-1}, 1 \times 10^{-3})$ is used instead.
While the noise added to the system is Gaussian, the state update between time steps is done using the nonlinear dynamics in equation \eqref{eq:dynamics}.
This results in a non-Gaussian distribution of the platoon state, which is meant to better approximate the noise of a real-world system.

\subsection{Dynamic Watermarking Setup}
\label{ssec:attack_detection}
A watermark as defined in \eqref{eq:control_input} was added to each vehicles input at each time step where
\begin{align}
     \Sigma_{e_i} = 0.25, \forall i \in \{1, \dots, \kappa\}.
     \label{eq:exp_watermark_cov}
\end{align}
While the watermark enables the detection of a wider range of attacks, it also increases the noise in the system resulting in reduced performance of the controller.
This leads to a trade off between performance and making sure the watermark is sufficiently visible in the face of other noise sources.
As mentioned in Section \ref{sec:int_platooning}, the benefit of V2V communications in vehicle platooning stem from the reduction in following distance under a constant spacing policy.
Therefore, to observe the reduction in performance due to the watermark, the mean and standard deviation of the bumper-to-bumper distances between vehicles were computed over 20 simulations with and without the added watermark.
Table \ref{tab:effect_of_watermark} shows the results for each level controller.
Since the level 1 controller is used only after an attack is detected, we do not add a watermark to this controller.
From this comparison, we note that there is indeed a reduction in performance resulting from the watermark as illustrated by the increased standard deviation for both the level 3 and level 2 controller.
However, even with this reduced performance, the level 3 and level 2 controller still maintain a smaller following distance than the level 1 controller.

\begin{table}[htb]
    \centering
    \begin{tabular}{c|c|c|c|c}
     &  \multicolumn{2}{c|}{Without watermark} & \multicolumn{2}{c}{With watermark} \\
    \hline   
    Level   & Mean & Std & Mean & Std  \\
    \hline 
    3      & 0.50 & 0.01 & 0.50 & 0.11 \\
    2      & 0.50 & 0.02 & 0.49 & 0.08 \\
    1      & 1.35 & 0.21 & - & -
    \end{tabular}
    \caption{
    Aggregate statistics for bumper-to-bumper distance (m) using 20 un-attacked simulations for each controller/watermark combination. Each simulation consists of a platoon of four vehicles following the trajectory in Figure \ref{fig:velocitymap}.
    }
    \label{tab:effect_of_watermark}
\end{table}

For each controller, the matrix normalization factor and the auto-correlation normalizing factor were generated from 50 simulations using \eqref{eq:sample_cov_estimate}-\eqref{eq:matrix_norm_factor} and \eqref{eq:auto_corr_ensemble}-\eqref{eq:auto_corr} respectively. 
To illustrate the benefit of using the matrix normalizing factor and auto-correlation normalizing factor of the proposed method, we provide a comparison to the LTI equivalent \cite{Hespanhol2018}.
We compute the matrix normalization factor and auto-correlation normalizing factor for the LTI case as follows
\begin{align}
    V_{(i,j)} &= \frac{1}{6001}\sum_{n=0}^{6000}V_{(i,j), n}, \\
    G_{(i,j)} &= I_{p_{(i,j)}+q_i},
\end{align}
where 6001 is the number of steps in the simulation.

\begin{figure}
    \centering
    \includegraphics[width=0.48\textwidth]{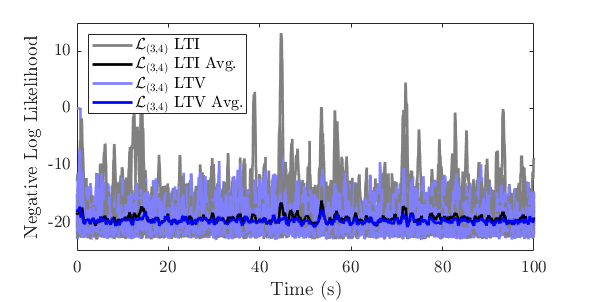}
    \caption{Comparing LTV to LTI Dynamic Watermarking}
    \label{fig:LTI_LTV_compare}
\end{figure}
For measurement $s_{(3,4)}$ we calculate the negative log likelihood using both the LTI and LTV normalizing factors for 20 un-attacked simulations as illustrated in Figure \ref{fig:LTI_LTV_compare}.
While the average negative log likelihood signal, taken over 20 simulations, is similar for both the LTI and LTV case, the actual signal for the LTI case shows far more spiking.
As a result, the threshold in the LTI case would likely need to be set higher to avoid false alarms.
However, a higher threshold reduces the ability of the detector to identify attacks.
Hence, the LTV Dynamic Watermarking is superior for this system.

To select a robust threshold, the attack thresholds are computed to achieve a desired false alarm rate based on a set of un-attacked trials.
The false alarm rate is defined as the number of time steps above the attack threshold divided by the total number of time steps. 
In this paper, 20 trials were used to calculate the thresholds which achieved a false alarm rate of 0.5\%.

To decide when to switch to the level 1 controller, we count the number of times each negative log likelihood has exceeded its threshold in the last 40 steps.
If this value exceeds 24 (60\%) for any given communication channel, the platoon switches to the level 1 controller.
The values of 60\% and 40 steps were chosen based on their ability to reduce the number of unnecessary switches from false alarms while still avoiding collisions in our simulated platoon.

\subsection{Attack Schemes}
\label{ssec:attack_scheme}

 \begin{figure*}[tbh]
    \centering
    \includegraphics[trim={0 1.3cm 0 1.5cm},clip,width=\textwidth]{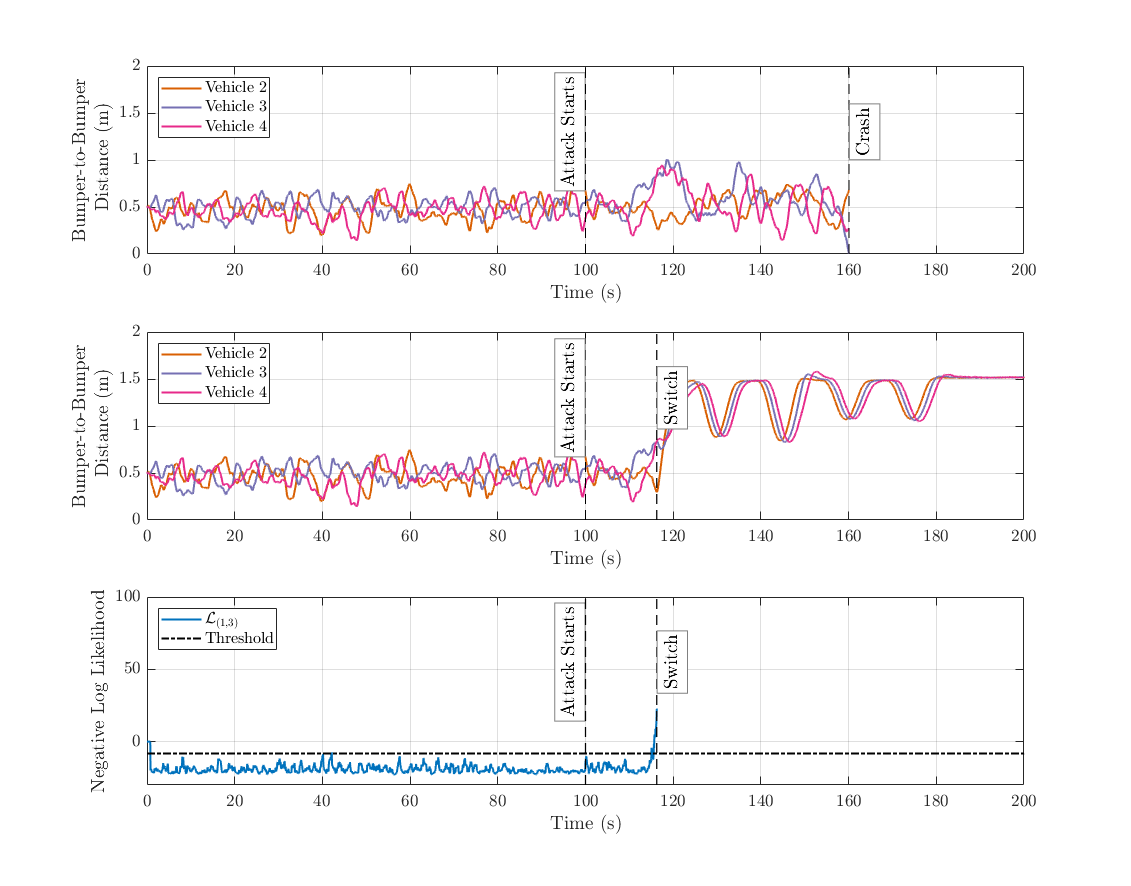}
    \caption{
    (top) Performance of the level 3 controller after a replay attack without switching to the level 1 controller and crashing soon thereafter.
    (middle) Platoon switching to level 1 controller after detecting the attack and safely completing the entire trajectory. 
    (bottom) Negative log likelihood of channel which detected the replay attack first.
    }
    \label{fig:results_fc_replay}
\end{figure*}

\begin{figure*}[tbh]
    \centering
    \includegraphics[trim={0 1.3cm 0 1.5cm},clip,width=\textwidth]{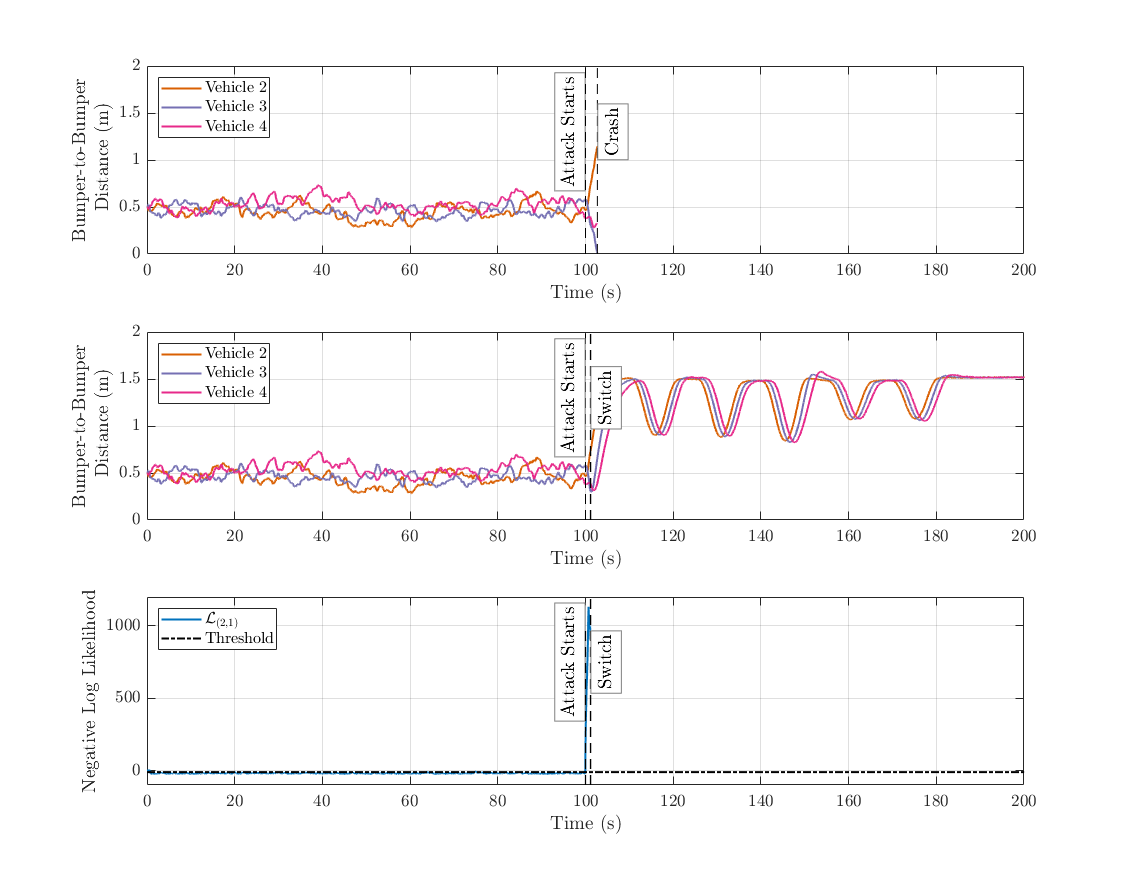}
    \caption{
    (top) Performance of the level 2 controller after an aggressive attack without switching to the level 1 controller and crashing soon thereafter.
    (middle) Platoon switching to level 1 controller after detecting the attack and safely completing the entire trajectory. 
    (bottom) Negative log likelihood of channel which detected the replay attack first.
    }
    \label{fig:results_lf_zero}
\end{figure*}

In this paper, we considered two different types of attacks, the stealthy replay attack and an aggressive attack.
While these attacks are not necessarily optimal, they are meant to represent two possible approaches of an attacker: 1) to go unnoticed while the effect of the attack slowly builds and 2) to make no attempt at remaining stealthy while trying to affect the system before a mitigation strategy can be implemented.

For a replay attack, the measurement signals (i.e. $s_{(i,j),n}$) from all or a subset of $H$ communication channels from an un-attacked realization are recorded and then played back when the simulation is run for a separate realization. 
Since an attack need not start at the beginning, we chose to start the attack 100s after the start of the simulation.
The replayed measurements included the distance between vehicles and the velocity of each vehicle in a platoon.

For the aggressive attack, we aimed to cause a collision as quickly as possible.
To generate the attack, the communication channel from vehicle 1 to vehicle 2 (i.e. $s_{(2,1),n}$) is hacked and the velocity measurement is set to zero. 
The attack leads vehicle 2 to believe the lead vehicle is braking and so brakes as well.
This results in vehicle 3, which receives the unaltered measurement from the lead vehicle, to crash into vehicle 2.
While this attack scheme is overt, it only requires intercepting a single communication channel and, if not mitigated quickly, results in a crash.
As with the replay attack, we start this attack 100s after the start of the simulation.

\subsection{Simulation Results}
To demonstrate the proposed detection algorithm, simulations were run for the level 2 and 3 control methods following Algorithm \ref{alg:cntrl} with an attack scheme as described in Section \ref{ssec:attack_scheme}. 
After an attack was detected as described in Section \ref{ssec:attack_detection}, the simulation was split into two concurrent simulations of the platoon, one in which the platoon degrades to the level 1 controller and one in which the platoon does not.
For the simulation, a crash was defined as the bumper-to-bumper distance between any two vehicles reaching 0 m.

For the simulations presented here, the replay attack involved attacking all communication channels between vehicles. 
However, even for simulations where a subset of channels were replay attacked, the attack was detected in the corresponding channels and the controller was successfully able to degrade.
This is important because even attacking a subset of communication channels can result in a crash and so being able to detect this in a timely manner is crucial.
In all cases, the attack was detected before any crash allowing the platoon to gracefully degrade to the level 1 controller.

For the level 3 controller, we see the effects of the replay attack in Figure \ref{fig:results_fc_replay}.
Even though the replay attack is subtle in operation, it is still able to cause a crash if the platoon does not degrade control schemes.
However, the level 3 controller appears resilient to the aggressive attack as illustrated in Figure \ref{fig:results_fc_zero} located in Appendix \ref{sec:sim_results_append}.
This is likely because the level 3 controller has a higher weighting on maintaining constant distance than velocity according to the controller gain in \eqref{eq:fully_K}.

In contrast, the level 2 controller appears to be more susceptible to the aggressive attack.
In Figure \ref{fig:results_lf_zero}, the performance of the level 2 controller worsens drastically under the aggressive attack as vehicle 2 brakes and almost immediately collides with vehicle 3 as a result.
However, the level 2 controller is more resilient to the replay attack, as seen in Figure \ref{fig:results_lf_replay} located in Appendix \ref{sec:sim_results_append}.
This is likely because the level 2 controller relies mainly on states that are measured directly.

\subsection{Scaling the platoon size}

To highlight the effect of scaling up the platoon size, we ran the same simulated experiments on a platoon of ten autonomous vehicles. 
In each realization, the platoon was subject to one of three possible attacks for both level 3 and 2 controllers: a replay attack on $30\%$ of communication channels, a replay attack on $100\%$ of channels, and an aggressive attack (as described in Section \ref{ssec:attack_scheme}).
For each realization, the time of attack was randomly sampled from a range of $[100, 200]$s and the subset of communication channels attacked are randomly selected to avoid any bias from the trajectory or specific communication channels.

For each controller, the matrix normalization factor and the auto-correlation normalizing factor were generated from 100 un-attacked realizations.
Attack thresholds are computed on 500 un-attacked realizations for a rate of false alarm of 0.1\%.
The platoon switches to the level 1 control strategy when the negative log likelihood exceeds the attack threshold for 18 steps in the last 40 steps (45\%) for any communication channel.

The results of comparing LTI vs LTV Dynamic Watermarking for the level 3 and level 2 controllers are shown in Tables \ref{tab:detection_stats_level3} and \ref{tab:detection_stats_level2} respectively.
The results are shown for replay attack 1/ replay attack 2/ aggressive attack.

\begin{table}[htb]
    \centering
    \begin{tabular}{|c|c|c|}
        \hline
                                & LTI                   & LTV               \\ \hline
        Successful detections   & 1513/1494/1530            & 1903/1876/1913        \\ \hline
        False alarms            & 482/470/470               & 88/78/87           \\ \hline
        Potential crashes       & \multicolumn{2}{c|}{1099/1997/24}             \\ \hline
        Actual crashes          & 4/36/0              & 8/46/0           \\ \hline
        No attack detected  & \multirow{2}{*}{1/0/0}    & \multirow{2}{*}{1/0/0} \\
        and no crash caused & & \\ \hline
        Mean detection & \multirow{2}{*}{2.771/2.823/0.905} & \multirow{2}{*}{2.701/2.587/0.957} \\ 
        time (s)      &  &\\ \hline
        Standard deviation   & \multirow{2}{*}{3.118/3.125/0.041}  & \multirow{2}{*}{3.812/3.576/0.363} \\ 
        detection time (s)        &           &  \\ \hline
    \end{tabular}
    \caption{Attack detection statistics for 2000 trials of level 3 controller.}
    \label{tab:detection_stats_level3}
\end{table}

\begin{table}[htb]
    \centering
    \begin{tabular}{|c|c|c|}
        \hline
                                & LTI                   & LTV               \\ \hline
        Successful detections   & 1880/1902/1968            & 1944/1987/1982        \\ \hline
        False alarms            & 103/98/32               & 21/13/18           \\ \hline
        Potential crashes       & \multicolumn{2}{c|}{3/4/2000}             \\ \hline
        Actual crashes          & 0/0/0               & 0/0/0           \\ \hline
        No attack detected  & \multirow{2}{*}{17/0/0}    & \multirow{2}{*}{35/0/0} \\
        and no crash caused & & \\ \hline
        Mean detection & \multirow{2}{*}{2.193/0.907/1.099} & \multirow{2}{*}{1.925/0.901/0.899} \\
        time (s)      &  &\\ \hline
        Standard deviation  & \multirow{2}{*}{10.221/0.075/0.031} & \multirow{2}{*}{9.003/0.077/0.022} \\ 
        detection time (s)        &           &  \\ \hline
    \end{tabular}
    \caption{Attack detection statistics for 2000 trials of level 2 controller.}
    \label{tab:detection_stats_level2}
\end{table}

From Tables \ref{tab:detection_stats_level3} and \ref{tab:detection_stats_level2}, potential crashes represent the number of realizations which would crash when not running the dynamic watermarking attack detection.
Looking at the potential crashes, we see that the level 3 controller is more robust to aggressive attacks whereas the level 2 controller is more robust to replay attacks.

Actual crashes represent the number of realizations which crashed while running the LTI or LTV dynamic watermarking algorithm. These crashes result from the effect of the attack on the platoon remaining below the user-defined threshold for detection. Selecting different user-defined parameters can reduce the number of actual crashes with the trade off of potentially higher number of false alarms. 

It is worth noting that for the level 3 controller, one replay attack 1 realization using the LTI attack detection scheme and one realization using the LTV scheme were not determined successful detections, false alarms, or crashes.
Upon further investigation, we concluded that the platoon performance was not affected by the attack in those realizations and so the dynamic watermark algorithm was not able to successfully detect the attack.
A similar conclusion was made, looking at replay attack 1 for the level 2 controller, for 17 and 35 realizations using the LTI and LTV attack detection schemes respectively, where the replay attack did not deteriorate the performance of the platoon significantly and so there were no crashes and no successful detection.
These statistics are recorded in Tables \ref{tab:detection_stats_level3} and \ref{tab:detection_stats_level2} as no attack detected and no crash caused.
Selecting different user-defined parameters may improve the detection rate for these realizations.

When comparing the LTV dynamic watermarking algorithm to the LTI version, we see that the LTV attack detection scheme has a greater number of successful attack detections and lesser number of false alarms while maintaining a similar number of crashes. 
This difference in performance is highlighted in Table \ref{tab:detection_stats_level3}, where successful attack detection using LTI dynamic watermarking is approximately 75\% whereas using LTV dynamic watermarking, we achieve approximately 95\%. 
Increasing the user-defined thresholds for the LTI attack detection scheme led to a marginal decrease in number of false alarms at the cost of a substantial increase in number of crashes. 
Hence, the LTV dynamic watermarking is superior for this system.
\section{Conclusion}\label{sec:conclusion}
In this paper, we formulated a linear-time varying version of dynamic watermarking for networked robotic systems and implemented it on a simulated platoon of autonomous vehicles.
We introduced different various levels of vehicle to vehicle communication and defined corresponding longitudinal controllers two of which leveraged the extra information for feedback control while the third does not rely on any vehicle to vehicle communication and was used as a mitigation strategy in the event of an attack.
While detailing the statistical tests, we provide implementation considerations on how to efficiently generate the normalization matrices used in the tests.
Compared to linear-time invariant dynamical watermarking, we showed that the method proposed in this paper is superior in that it provides a more consistent test metric.
We described two different attack schemes, one stealthy and one aggressive, and showed that our algorithm could detect both types of attack while utilizing each controller and successfully degrade to a safe control strategy before a crash can occur.

\renewcommand{\bibfont}{\normalfont\small}
\printbibliography

\appendix
\subsection{Derivation of LTV Dynamics}\label{sec:derive}
This section provides a thorough derivation of the lateral and longitudinal dynamics used in this paper.
\begin{table}[b]
    \centering
    \begin{tabular}{| c | c |}
        \hline
        \textbf{Constant} & \textbf{Value}\\\hline\hline
        $c_1$ &  $ 1.6615\times10^{-5}$\tabspacer\\\hline
        $c_2$ &  $-1.9555\times10^{-7}$\tabspacer\\\hline
        $c_3$ &  $ 3.619 \times10^{-6}$\tabspacer\\\hline
        $c_4$ &  $ 4.382 \times10^{-7}$\tabspacer\\\hline
        $c_5$ &  $-8.1112\times10^{-2}$\tabspacer\\\hline
        $c_6$ &  $-1.4736\times10^{ 0}$\tabspacer\\\hline
        $c_7$ &  $ 1.2569\times10^{-1}$\tabspacer\\\hline
        $c_8$ &  $ 7.6459\times10^{-2}$\tabspacer\\\hline
        $c_9$ &  $-1.3991\times10^{-2}$\tabspacer\\\hline
    \end{tabular}
    \caption{Fitted constants for the nonlinear dynamics in Eq. \eqref{eq:dynamics}}
    \label{tab:dyn_const}
\end{table}
Let $\{\bar{x}_{i}(t),\bar{y}_{i}(t),\bar{\psi}_{i}(t),\bar{v}_{i}(t),\bar{v}_{i}^d(t),\bar{\delta}_{i}(t)\}$ be the continuous states and inputs that define the trajectory for vehicle $i$,
and let $\{\bar{x}_{i,n},\bar{y}_{i,n},\bar{\psi}_{i,n},\bar{v}_{i,n},\bar{v}_{i,n}^d,\bar{\delta}_{i,n}\}_{n=0}^\infty$ be the corresponding sampled trajectory for the discretized system.
To ease notation we drop the $(t)$ from the continuous states.

\subsubsection{Lateral Dynamics}
Consider the lateral error $\Delta \text{lat}_i$ and heading error $\Delta \psi_i$ defined as $\Delta \text{lat}_i = (y_i-\bar{y}_i)\cos(\bar{\psi}_i)-(x_i-\bar{x}_i)\sin(\bar{\psi}_i)$ and $\Delta \psi_i=\psi_i-\bar{\psi}_i$.
We approximate the continuous dynamics as follows.
\begin{align}
    \Delta\dot{\psi}_i&=\dot{\psi}-\dot{\bar{\psi}}=\frac{\tan(c_1\delta_i+c_2)v_i}{c_1+c_4v_i^2}-\frac{\tan(c_1\bar{\delta}_i+c_2)\bar{v}_i}{c_3+c_4\bar{v}_i^2}=\nonumber\\
    &\approx\frac{(\tan(c_1\delta_i +c_2)
    -\tan(c_1\bar{\delta}_i+c_2)) \bar{v}_i}{c_3 + c_4\bar{v}_i^{2}}\label{eq:nl_heading_error},
\end{align}
where the second equality comes from \eqref{eq:dynamics} and the approximation from $v_i\approx\bar{v_i}$ and
\begin{align}
    &\dot{\Delta\text{lat}}_i=(\dot{y}_i-\dot{\bar{y}}_i)\cos(\bar{\psi}_i)-(\dot{x}_i-\dot{\bar{x}}_i)\sin(\bar{\psi}_i)+\nonumber\\
    &\qquad-\dot{\bar{\psi}}_i\left((x_i\bar{x}_i)\cos(\bar{\psi}_i)+(y)_i\bar{y}_i)\sin(\bar{\psi}_i)\right)=\nonumber\\
    &\approx v_i\sin(\Delta\psi_i)+\dot{\psi}_i(c_8+c_9v_i^2)\cos(\Delta\psi_i)+\nonumber\\
    &\qquad-\dot{\bar{\psi}}_i(c_8+c_9\bar{v}_i^2)=\nonumber\\
    &\approx\bar{v}_i\sin(\Delta\psi_i)+\frac{(c_8+c_9\bar{v}_i^2)\bar{v}_i}{c_3 + c_4\bar{v}_i^{2}}\times\nonumber\\
    &\qquad\times(\tan(c_1\delta_i +c_2)\cos(\Delta\psi_i)-\tan(c_1\bar{\delta}_i+c_2))\label{eq:nl_lat_error},
\end{align}
where the first inequality comes from \eqref{eq:dynamics} and $(x_i-\bar{x}_i)\cos(\bar{\psi}_i)+(y_i-\bar{y}_i)\sin(\bar{\psi}_i)\approx0$, and the second from \eqref{eq:dynamics} and $v_i\approx\bar{v_i}$.
These approximations are reasonable since the longitudinal controller aims to maintain the desired velocity and we reduce longitudinal error defined as  $(x_i-\bar{x}_i)\cos(\bar{\psi}_i)+(y_i-\bar{y}_i)\sin(\bar{\psi}_i)$ by accommoding drift along the trajectory.
Linearizing \eqref{eq:nl_heading_error}-\eqref{eq:nl_lat_error} then gives us
\begin{align}
    &\begin{bmatrix}
    \dot{\Delta \text{lat}}_i\\
    \dot{\Delta\psi}_i
    \end{bmatrix}=
    \begin{bmatrix}
    0&\bar{v}_i\\0&0
    \end{bmatrix}
    \begin{bmatrix}
    \Delta \text{lat}_i\\
    \Delta\psi_i
    \end{bmatrix}
    +\nonumber\\&\quad
    +\frac{c_1 \bar{v}_i}
    {\cos^2(c_1\bar{\delta}_i+c_2)(c_3+c_4\bar{v}_i^2)}\begin{bmatrix}
    (c_8+c_9\bar{v}_i^2)\\
    1
    \end{bmatrix}
    \Delta \delta_i,
\end{align}
where $\Delta \delta_i=\delta_i-\bar{\delta}_i$.
Discretizing using a step size of 0.05 s and a zero-order hold on $\bar{v}_i$ and $\delta_i$ then results in
\begin{align} \label{eq:lateral_dynamics}
    \begin{bmatrix}
    \Delta \text{lat}_{i,n+1}\\
    \Delta\psi_{i,n+1}
    \end{bmatrix}&=
    \begin{bmatrix}
    1&\frac{\bar{v}_{i,n}}{20}\\0&1
    \end{bmatrix}
    \begin{bmatrix}
    \Delta \text{lat}_{i,n}\\
    \Delta\psi_{i,n}
    \end{bmatrix}+\nonumber\\
    &\quad+\frac{c_1 \bar{v}_{i,n}}
    {\cos^2(c_1\bar{\delta}_{i,n}+c_2)(c_3+c_4\bar{v}_{i,n}^2)}\times\nonumber\\
    &\quad\times\begin{bmatrix}
    \left(\frac{(c_8+c_9\bar{v}_{i,n}^2)}{20}+\frac{\bar{v}_{i,n}}{800}\right)\\
    \frac{1}{20}
    \end{bmatrix}
    \Delta \delta_{i,n}.
\end{align}

\subsubsection{Longitudinal Dynamics}
The state of the longitudinal dynamics for a platoon of $\kappa$ vehicles, as illustrated in Figure \ref{fig:platoon}, is made up of each vehicle's velocity $v_{1,n},\hdots,v_{\kappa,n}$, and the distances between subsequent vehicles $d_{1,n},\hdots,d_{\kappa-1,n}$.
Next, under the assumption that the tracking error for the lane keeping controller is sufficiently small, we linearize the longitudinal dynamics from \eqref{eq:dynamics} as
\begin{align}
    \dot{\Delta d}_i= \Delta v_i-\Delta v_{i+1}\\
    \dot{\Delta v}_i=\alpha_i\Delta v_{i}-\alpha_i \Delta v_i^d,
\end{align}
where $\Delta d_i=d_i-\bar{d}_i$, $\Delta v_i=v_i-\bar{v}_i$, $\Delta v_i^d=v_i^d-\bar{v}_i^d$, and $\alpha_i$ is the continuous equivalent to \eqref{eq:alpha} defined as
\begin{align}
    \alpha_i=c_6+2c_7(\bar{v}_i-\bar{v}_i^d).
\end{align}
Selecting a time step of 0.05 s and assuming a zero-order hold on the input, these dynamics are then discretized as
\begin{align}
    \Delta d_{i,n+1}=&\Delta d_{i,n}+\frac{\beta_i-1}{\alpha_i}\Delta v_{i,n}-\frac{\beta_{i+1}-1}{\alpha_{i+1}}\Delta v_{i+1,n}+\nonumber\\
    &\hspace{-1.4cm}-(\frac{\beta_{i}-1}{\alpha_i}-0.05)\Delta v_i^d+(\frac{\beta_{i+1}-1}{\alpha_{i+1}}-0.05)\Delta v_{i+1}^d\\
    \Delta v_{i,n+1}=&\beta_i \Delta v_{i,n} -(\beta_i-1) \Delta v_{i,n}^d, \label{eq:v_discrete}
\end{align}
where 
\begin{align}
    \alpha_{i,n}&=c_6+2c_7(\bar{v}_{i,n}-\bar{v}_{i,n}^d),\label{eq:alpha}\\
    \beta_{i,n}&=e^{0.05\alpha_{i,n}}\label{eq:beta}.
\end{align}

Vectorizing these discrete dynamics for the state vector $x_n=[\Delta d_{1,n}~\cdots~\Delta d_{\kappa-1,n}~\Delta v_{1,n}~\cdots~\Delta v_{\kappa,n}]^\T$ and inputs $u_{i,n}=\Delta v_{i,n}^d$ results in an LTV system satisfying \eqref{eq:state_update} where 
\begin{align}
    A_n&=\nonumber\\
    &\left[\begin{array}{c;{2pt/2pt}c}
    I_{\kappa-1} & \begin{matrix}\frac{\beta_{1,n}-1}{\alpha_{1,n}}&-\frac{\beta_{2,n}-1}{\alpha_{2,n}}&&\\&\ddots&\ddots&\\&&\frac{\beta_{\kappa-1,n}-1}{\alpha_{\kappa-1,n}}&-\frac{\beta_{\kappa,n}-1}{\alpha_{\kappa,n}}\end{matrix}\\\hdashline[2pt/2pt]
    0_{\kappa\times\kappa-1}&\text{diag}(\beta_{1,n},\hdots,\beta_{\kappa,n})
    \end{array}\right], \label{eq:long_A}
\end{align}
\begin{align}
    B_{i,n}=\begin{cases}
    \begin{bmatrix}\frac{1}{20}-\frac{\beta_{1,n}-1}{\alpha_{1,n}}\\0_{\kappa-2\times1}\\1-\beta_{1,n}\\0_{\kappa-1\times1}\end{bmatrix} & i=1\\\noalign{\vskip3pt}
    \begin{bmatrix}0_{i-2\times1}\\\frac{\beta_{i,n}-1}{\alpha_{i,n}}-\frac{1}{20}\\ \frac{1}{20}-\frac{\beta_{i,n}-1}{\alpha_{i,n}}\\0_{\kappa-2\times1}\\1-\beta_{i,n}\\0_{ \kappa-i\times1}\end{bmatrix}& i\neq1.
    \end{cases} \label{eq:long_B}.
\end{align}

The measurement model satisfies \eqref{eq:measurement_model} where the matrix $C$ takes the form 
\begin{align}
    C_{i}=\begin{cases}
    \begin{bmatrix}
    0_{1\times \kappa-1}&1&0_{1\times\kappa-1}
    \end{bmatrix}
    &i=1\\\noalign{\vskip3pt}
    \left[\begin{array}{c;{2pt/2pt}c;{2pt/2pt}c;{2pt/2pt}c;{2pt/2pt}c}
    0_{2\times i-2}&\begin{matrix}1\\0\end{matrix}&0_{2\times\kappa-1}&\begin{matrix}0\\1\end{matrix}&0_{2\times\kappa-i}
    \end{array}\right]
    &i\neq1.
    \end{cases} \label{eq:C1}
\end{align}
for the level 3 and level 1 controllers defined in Section \ref{sec:design} and takes the form
\begin{align}
    C_{i}=\begin{cases}
    \left[\begin{array}{c;{2pt/2pt}c;{2pt/2pt}c}
    \begin{matrix}
    1&0&\hdots&0\\
    \vdots&\ddots&\ddots&\vdots\\
    1&\hdots&1&0\\
    1&\hdots&1&1
    \end{matrix} & 
    0_{\kappa-1\times1}&0_{\kappa-1}\\\hdashline[2pt/2pt]
    0_{1\times\kappa-1}&1&0_{1\times\kappa-1}
    \end{array}\right]
    &i=1\\\noalign{\vskip3pt}
    \left[\begin{array}{c;{2pt/2pt}c;{2pt/2pt}c;{2pt/2pt}c;{2pt/2pt}c}
    0_{2\times i-2}&\begin{matrix}1\\0\end{matrix}&0_{2\times\kappa-1}&\begin{matrix}0\\1\end{matrix}&0_{2\times\kappa-i}
    \end{array}\right]
    &i\neq1.
    \end{cases}\label{eq:C2}
\end{align}
for the level 2 controller defined in Section \ref{sec:design}.

\subsection{Closed Loop System Matrices}\label{sec:cl_mat}
The matrices for the closed loop dynamics in \eqref{eq:closed_loop} satisfy
\begin{align}
    \bar{A}_n&=\left[\begin{array}{c;{2pt/2pt}c}
    A_n & 
    B_{1,n}K_{1,n}~\cdots~B_{\kappa,n}K_{\kappa,n}\\\hdashline[2pt/2pt]
    \begin{matrix}
    -\sum_{j\in H_1} L_{1,j}C_j\\
    \vdots\\
    -\sum_{j\in H_\kappa}L_{\kappa,j}C_j
    \end{matrix}&
    \text{diag}(M_{1,n},\hdots,M_{\kappa,n})
    \end{array}\right],\\
    \bar{B}_n&=\left[\begin{array}{c}
    \begin{matrix}B_{1,n} & \cdots & B_{\kappa,n}\end{matrix}\\\hdashline[2pt/2pt]
    \text{diag}(N_1B_{1,n},\hdots,N_\kappa B_{\kappa,n})
    \end{array}\right],\\
    \bar{L}_n&=\left[\begin{array}{c}0_{p\times\left(\kappa\sum_{j=1}^\kappa r_j\right)}\\\hdashline[2pt/2pt]
    \text{diag}(L_{1,n},\hdots,L_{\kappa,n})\end{array}\right],\\
    L_{i,n}&=\begin{bmatrix}L_{(i,1),n}&\cdots&L_{(i,\kappa),n}\end{bmatrix}\label{eq:L_concat}.
\end{align}
In \eqref{eq:L_concat}, when $(i,j)\notin H$, $L_{(i,j),n}$ is replaced by $0_{p_i\times m_i}$.

\subsection{Lateral Controller and Observer}\label{sec:lat_control}
For lane keeping, a lateral controller is introduced which operates independently of the longitudinal controller (i.e. each vehicle runs the same lateral controller and observer at all times).
The feedback law follows
\begin{align}
    \delta_{i,n}=\begin{bmatrix} -0.25 & -1\end{bmatrix}\begin{bmatrix}\Delta \hat{\text{lat}}_{i,n} \\ \Delta \hat{\psi}_{i,n} \end{bmatrix} \label{eq:lat_controller}.
\end{align}
Furthermore, the observer follows
\begin{align}
    \begin{bmatrix}
        \Delta \hat{\text{lat}}_{i,n+1} \\ 
        \Delta \hat{\psi}_{i,n+1} 
    \end{bmatrix}   &=  
    \Bigg(
        \begin{bmatrix}
            1   &   \frac{\bar{v}_{i,n}}{20} \\
            0   &   1
        \end{bmatrix}+
        \begin{bmatrix}
            0.3 & \frac{\hat{v}_{i,n}}{20} \\
            0 & 0.2
        \end{bmatrix}
    \Bigg)
    \begin{bmatrix}
        \Delta \hat{\text{lat}}_{i,n} \\ 
        \Delta \hat{\psi}_{i,n} 
    \end{bmatrix} + \nonumber \\
    +&\frac{c_1 \bar{v}_{i,n}}
    {\cos^2(c_1\bar{\delta}_{i,n}+c_2)(c_3+c_4\bar{v}_{i,n}^2)}\times\nonumber\\
    \times&\begin{bmatrix}
        \left(\frac{(c_8+c_9\bar{v}_{i,n}^2)}{20}+\frac{\bar{v}_{i,n}}{800}\right)\\
        \frac{1b_{i,n}}{20}
    \end{bmatrix}
    \Delta \delta_{i,n} \nonumber \\
    -& 
    \begin{bmatrix}
        0.3 & \frac{\hat{v}_{i,n}}{20} \\
        0 & 0.2
    \end{bmatrix}
    \Delta \text{y-lat}_{i,n},
\end{align}
where $\Delta \text{y-lat}_{i,n}$ is the measurement of the lateral and heading error.

\begin{table*}[t]
    \centering
    \begin{tabular}{|c|c|c|}\hline
        Controller & $N_i$ & $M_{i,n}$ \\\hline\hline
        Level 3 & \parbox{2.8in}{\begin{align}N_i=I_{2\kappa-1}\end{align}} & \parbox{3in}{\begin{align}M_{i,n}=A_n+\sum_{j=1}^\kappa \left(B_{j,n}K_{j,n}+L_{(i,j),n}C_{j}\right)\label{eq:fully_M}\end{align}} \\\hline
    Level 2 & \parbox{2.8in}{\begin{align}
    &N_ix_n=\nonumber\\&=
    \begin{cases}
        \begin{bmatrix}
            d_{1,n} & v_{1,n} & v_{2,n}
        \end{bmatrix}^\T & i=1\\\noalign{\vskip3pt}
        \begin{bmatrix}
            d_{1,n} & d_{2,n} & v_{1,n} & v_{2,n} & v_{3,n}
        \end{bmatrix}^\T & i=2\\\noalign{\vskip3pt}
        \begin{bmatrix} 
            \underline{d}_{\kappa,n} & d_{\kappa-1,n}&v_{1,n} & v_{\kappa-1,n} & v_{\kappa,n}
        \end{bmatrix}^\T & i=\kappa\\\noalign{\vskip3pt}
        \Big[\begin{matrix}
            \underline{d}_{i,n} & d_{i-1,n} & d_{i,n}
        \end{matrix}&~\\
        \qquad\qquad\begin{matrix} v_{1,n} & v_{i-1,n} & v_{i,n} & v_{i+1,n}
        \end{matrix}\Big]^\T & \text{o/w}
    \end{cases}
\end{align}} & \parbox{3.2in}{
\begin{align}
    &M_{i,n}=N_iB_{i,n}K_{i,n}+\nonumber\\
    &+\begin{cases}
    \left[\begin{array}{c;{2pt/2pt}c} 0.5 & \begin{matrix}\sigma_{1,n}& -\sigma_{2,n}\end{matrix}\\\hdashline[2pt/2pt]0_{2\times1} & \text{diag}(\theta_{1,n},\theta_{2,n})\end{array}\right] & i=1\\\noalign{\vskip3pt}
    \left[\begin{array}{c;{2pt/2pt}c} 0.5I_2 & \begin{matrix}\sigma_{1,n}& -\sigma_{2,n} & 0\\0 & \sigma_{2,n} & -\sigma_{3,n}\end{matrix}\\\hdashline[2pt/2pt]0_{3\times2} & \text{diag}(\theta_{1,n},\theta_{2,n},\theta_{3,n})\end{array}\right] & i=2\\\noalign{\vskip3pt}
    \left[\begin{array}{c;{2pt/2pt}c} 0.5I_2 & \begin{matrix}\sigma_{1,n}& 0 & -\sigma_{\kappa,n}\\0 & \sigma_{\kappa-1,n} & -\sigma_{\kappa,n}\end{matrix}\\\hdashline[2pt/2pt]0_{3\times2} & \text{diag}(\theta_{1,n},\theta_{\kappa-1,n},\theta_{\kappa,n})\end{array}\right]& i=\kappa\\\noalign{\vskip3pt}
    \left[\begin{array}{c;{2pt/2pt}c} 0.5I_3 & \begin{matrix} \sigma_{1,n}& 0 & -\sigma_{i,n} & 0\\ 0 & \sigma_{i-1,n} & -\sigma_{i,n} & 0\\ 0 & 0 & \sigma_{i,n} & -\sigma_{i+1,n} \end{matrix}\\\hdashline[2pt/2pt] 0_{4\times3} & \text{diag}(\theta_{1,n},\theta_{i-1,n},\theta_{i,n},\theta_{i+1,n}) \end{array}\right]& \text{o/w}
    \end{cases}\label{eq:leader_M}
\end{align}} 
\\\hline
Level 1 & \parbox{2.8in}{\begin{align}
    N_ix_{i,n}=\begin{cases}
    \begin{bmatrix}v_{1,n}\end{bmatrix} & i=1\\\noalign{\vskip3pt}
    \begin{bmatrix}d_{i-1,n}\\v_{i-1,n}\\v_{i,n}\end{bmatrix} & i\neq1
    \end{cases}
\end{align}} & \parbox{3in}{\begin{align}
    M_{i,n}&=N_iB_{i,n}K_{i,n}+\nonumber\\
    &+\begin{cases}
    \theta_{1,n} & i=1\\
    \begin{bmatrix}0.5 & \frac{\beta_{i-1,n}-1}{\alpha_{i-1,n}} & -\sigma_{i,n}\\
    -1.2 & \beta_{i-1,n} & 0\\0 & 0 & \theta_{i,n}\end{bmatrix} & i\neq1
    \end{cases}\label{eq:acc_M}
\end{align}}
\\\hline
    \end{tabular}
    \caption{Linear transform and observer system matrices for each controller}
    \label{tab:obs_table1}
\end{table*}

\subsection{Sufficient conditions for Correlation}\label{sec:corr_prop}
The following proposition provides a sufficient condition for a non-zero correlation.
\begin{prop} 
Consider a closed loop LTV system satisfying \eqref{eq:closed_loop}. If for some $\rho_{(i,j)}\in\mathbb{N}$
\begin{align}
    &\begin{bmatrix}
    W_{(i,j)}C_{j} & 0_{m_i \times t}
    \end{bmatrix}
    \bar{A}_{n-1}\cdots
    \bar{A}_{n-\rho_{(i,j)}}
    \bar{B}_{n-\rho{(i,j)}-1}\times\nonumber\\
    &~\times
    \begin{bmatrix}
    0_{q_i \times \left(p+\sum_{j=1}^{i-1}p_j\right)} & 
    I_{q_i} & 
    0_{q_i\times \left(\sum_{j=i+1}^\kappa p_j\right)}
    \end{bmatrix}^\T
    \neq0_{p_{(i,j)}\times q_i}\label{eq:single_step_corr_cond}
\end{align}
then 
\begin{align}
    \mathds{E}\left[W_{(i,j)}s_{(i,j),n}e_{i,n-\rho_{(i,j)}-1}^\T\right]\neq 0_{p_{(i,j)}\times q_i}\label{eq:single_step_corr}
\end{align}
\end{prop}
\begin{proof}
Due to the assumption that the watermark is mutually independent and zero mean, only terms that are a function of the watermark at that particular step have non-zero expectation. For ease of notation, these terms are removed without loss of generality in this proof.
Expanding the communicated measurement on the left side of \eqref{eq:single_step_corr} first using \eqref{eq:com_measurement} then \eqref{eq:closed_loop} results in 
\begin{align}
    \mathds{E}&\left[W_{(i,j)}s_{(i,j),n}e_{i,n-\rho_{(i,j)}-1}^\T\right]=\nonumber\\
    &=\mathds{E}\left[W_{(i,j)}C_{j}x_{n}e_{i,n-\rho_{(i,j)}-1}^\T\right]=\nonumber\\
    &=\mathds{E}\Big[\begin{bmatrix}
    W_{(i,j)}C_{j} & 0_{m_i \times t}
    \end{bmatrix}\bar{A}_{n-1}\cdots\bar{A}_{n-\rho_{(i,j)}}\times\nonumber\\
    &~~\quad\times\bar{B}_{n-\rho_{(i,j)}-1}e_{n-\rho_{(i,j)}-1}e_{i,n-\rho_{(i,j)}-1}^\T\Big]=\nonumber\\
    &=\begin{bmatrix}
    W_{(i,j)}C_{j} & 0_{m_i \times t}
    \end{bmatrix}
    \bar{A}_{n-1}\cdots
    \bar{A}_{n-\rho_{(i,j)}}
    \bar{B}_{n-\rho_{(i,j)}-1}\times\nonumber\\
    &~~\quad\times
    \begin{bmatrix}
    0_{q_i \times \left(p+\sum_{j=1}^{i-1}p_j\right)} & 
    I_{q_i} & 
    0_{q_i\times \left(\sum_{j=i+1}^\kappa p_j\right)}
    \end{bmatrix}^\T\Sigma_{e_i}.
\end{align}
Since $\Sigma_{e_i}$ is full rank and \eqref{eq:single_step_corr_cond} holds, \eqref{eq:single_step_corr} holds as well.
\end{proof}



\subsection{Longitudinal Observer Matrices}\label{sec:obs_mat}
\begin{table*}[!htb]
    \centering
    \begin{tabular}{|c|c|c|}\hline
        Controller & $L_{(i,j),n}$ & $W_{(i,j)},~U_{(i,j)}$ \\\hline\hline
        Level 3 &  
        \parbox{3.2in}{\begin{align}
        L_{(i,j),n} = 
        \begin{cases}
            \begin{bmatrix}
                -0.05 & 0_{1 \times \kappa-2} & -0.1 & 0_{1 \times \kappa-1}
            \end{bmatrix}^{\T} & j = 1 \\\noalign{\vskip3pt}
            \Big[
                \begin{array}{c;{2pt/2pt}c;{2pt/2pt}c;{2pt/2pt}c}
                    0_{2\times \kappa-2} &
                    \begin{matrix}
                        -0.5 \\
                        0.05
                    \end{matrix} & 
                    0_{2\times\kappa-1} &
                    \begin{matrix} 
                        0 \\
                        -0.1 
                    \end{matrix}
                \end{array}
            \Big]^{\T} & j = \kappa\\\noalign{\vskip3pt}
            \Big[\begin{array}{c;{2pt/2pt}c;{2pt/2pt}c;{2pt/2pt}}
                    0_{2\times j-2} &
                    \begin{matrix}
                        -0.5 \\
                        0.05
                    \end{matrix} & 
                    \begin{matrix}
                        0 \\
                        -0.05
                    \end{matrix}           
                \end{array} & ~\\
                \quad \begin{array}{c;{2pt/2pt}c;{2pt/2pt}c}
                    0_{2\times\kappa-2} &
                    \begin{matrix} 
                        0 \\
                        -0.1 
                    \end{matrix} & 
                    0_{2\times\kappa-j}
                \end{array}
            \Big]^{\T} & \text{o/w}.
        \end{cases}\label{eq:fully_L}
    \end{align}} & \parbox{2.8in}{\begin{align}W_{(i,j)}=C_j,~U_{(i,j)}=I_{m_j}\end{align}} \\\hline
    Level 2 & \parbox{3.2in}{
    \begin{align}
    L_{(1,j),n}&=\begin{cases}
    \left[\begin{array}{c;{2pt/2pt}c}0_{3\times\kappa-1} & \begin{matrix}-0.05\\-0.1\\0\end{matrix}\end{array}\right] & j=1\\\noalign{\vskip3pt}
    \begin{bmatrix}-0.5 & 0 & 0\\0.05 & 0 & -0.1\end{bmatrix}^\T & j=2
    \end{cases}\label{eq:leader_L1}
\end{align}
for the first vehicle,
\begin{align}
    L_{(2,j),n}&=\begin{cases}
    \left[\begin{array}{c;{2pt/2pt}c;{2pt/2pt}c}\begin{matrix}-0.25\\0_{4\times1}\end{matrix}&0_{1\times\kappa-2} & \begin{matrix}-0.05\\0\\-0.1\\0_{2\times1}\end{matrix}\end{array}\right] & j=1\\\noalign{\vskip3pt}
    \begin{bmatrix}-0.25 & 0 & 0 & 0 & 0\\0.05 & -0.05 & 0 & -0.1 & 0\end{bmatrix}^\T & j=2\\\noalign{\vskip3pt}
    \begin{bmatrix}0 & -0.5 & 0 & 0 & 0\\0 & 0.05 & 0 & 0 & -0.1\end{bmatrix}^\T & j=3
    \end{cases}
\end{align}
for the second vehicle,
\begin{align}
    L_{(\kappa,j),n}&=\begin{cases}
    \left[\begin{array}{c;{2pt/2pt}c}0_{5\times\kappa-2} & \begin{matrix}-0.5 & -0.05\\0 & 0\\0 & 0.1\\0_{2\times1} & 0_{2\times1}\end{matrix}\end{array}\right] & j=1\\\noalign{\vskip3pt}
    \begin{bmatrix}0 & -0.05 & 0 & -0.1 & 0\end{bmatrix}^\T & j=\kappa-1\\\noalign{\vskip3pt}
    \begin{bmatrix}0 & -0.5 & 0 & 0 & 0\\ 0.05 & 0.05 & 0 & 0 & -0.1\end{bmatrix}^\T & j=\kappa
    \end{cases}
\end{align}
for the last vehicle, and
\begin{align}
    &L_{(i,j),n}=\nonumber\\
    &=\begin{cases}
    \left[\begin{array}{c;{2pt/2pt}c;{2pt/2pt}c;{2pt/2pt}c}0_{7\times i-2} & \begin{matrix}-0.5\\0_{6\times1}\end{matrix}&0_{7\times\kappa-i} & \begin{matrix}-0.05\\0\\0\\-0.1\\0_{3\times1}\end{matrix}\end{array}\right] & j=1\\\noalign{\vskip3pt}
    \begin{bmatrix}0 & -0.05 & 0 & 0 & -0.1 & 0 & 0 \end{bmatrix}^\T & j=i-1\\\noalign{\vskip3pt}
    \begin{bmatrix}0 &- 0.5 & 0 & 0 & 0 & 0 & 0\\ 0.05 & 0.05 & -0.05 & 0 & 0 & -0.1 & 0\end{bmatrix}^\T & j=i\\\noalign{\vskip3pt}
    \begin{bmatrix}0 & 0 & -0.5 & 0 & 0 & 0 & 0\\ 0 & 0 & 0.05 & 0 & 0 & 0 & -0.1\end{bmatrix}^\T & j=i+1\\
    \end{cases}\label{eq:leader_L4}
\end{align}
for the remaining vehicles i.e. $i\notin\{1,2,\kappa\}$.} & \parbox{2.8in}{\begin{align}
    &U_{(i,j)}N_{i}x_n=W_{(i,j)}C_{j}x_n=\nonumber\\
    &\quad\qquad=\begin{cases}
        \begin{bmatrix}v_{1,n}\end{bmatrix} & j=1,~i=1\\\noalign{\vskip3pt}
        \begin{bmatrix}\sum_{k=1}^{i-1}d_{k,n}&v_{1,n}\end{bmatrix}^\T & j=1,~i\neq1\\\noalign{\vskip3pt}
        \begin{bmatrix}v_{j,n}\end{bmatrix} & j\neq1,~i=j+1\\\noalign{\vskip3pt}
        \begin{bmatrix}d_{j-1,n}&v_j\end{bmatrix}^\T & j\neq1,~i\neq j+1
    \end{cases}.
\end{align}}\\\hline
Level 1 & \parbox{3.2in}{\begin{align}
    L_{(i,i),n} = 
    \begin{cases}
        -0.1 & i = 1 \\
        \begin{bmatrix}
            -0.5 & 0.05 \\
            -1.2 & 0\\
            0 & -0.1
        \end{bmatrix}
        & i \neq 1
    \end{cases}\label{eq:acc_L}
\end{align}} & N/A\\\hline
    \end{tabular}
    \caption{Observer gain and measurement transform matrices}
    \label{tab:obs_table2}
\end{table*}

This section provides the observer matrices for all three control levels in Tables \ref{tab:obs_table1} and \ref{tab:obs_table2}. To ease notation we define the following
\begin{align}
    \theta_{i,n}&=\beta_{i,n}-0.1,\label{eq:theta}\\
    \sigma_{i,n}&=\frac{\beta_{i,n}-1}{\alpha_{i,n}}-0.05.\label{eq:sigma}
\end{align}

\subsection{Simulation results}
\label{sec:sim_results_append}

The results of the aggressive attack on the level 3 controller and the replay attack on the level 2 controller, shown in Figures \ref{fig:results_fc_zero} and \ref{fig:results_lf_replay} respectively.
\clearpage 
\begin{figure*}[!h]
    \centering
    \includegraphics[trim={0 1.3cm 0 1.5cm},width=\textwidth]{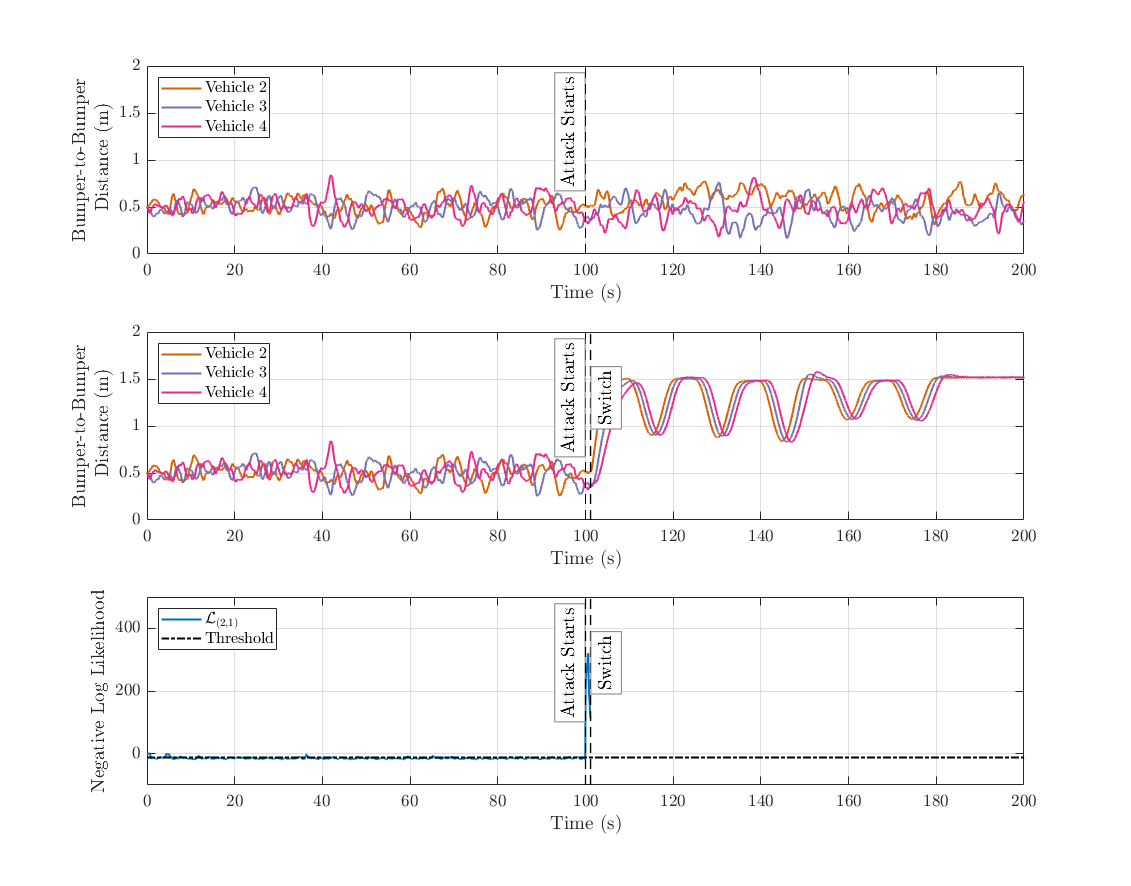}
    \caption{
    (top) Performance of the level 3 controller after an aggressive attack without switching to the level 1 controller.
    However, it completes the trajectory without crashing.
    (middle) Platoon switching to level 1 controller after detecting the attack and completing the entire trajectory. 
    (bottom) Negative log likelihood of channel which detected the replay attack first.
    }
    \label{fig:results_fc_zero}
\end{figure*}

\begin{figure*}[h]
    \centering
    \includegraphics[trim={0 1.3cm 0 1.5cm},width=\textwidth]{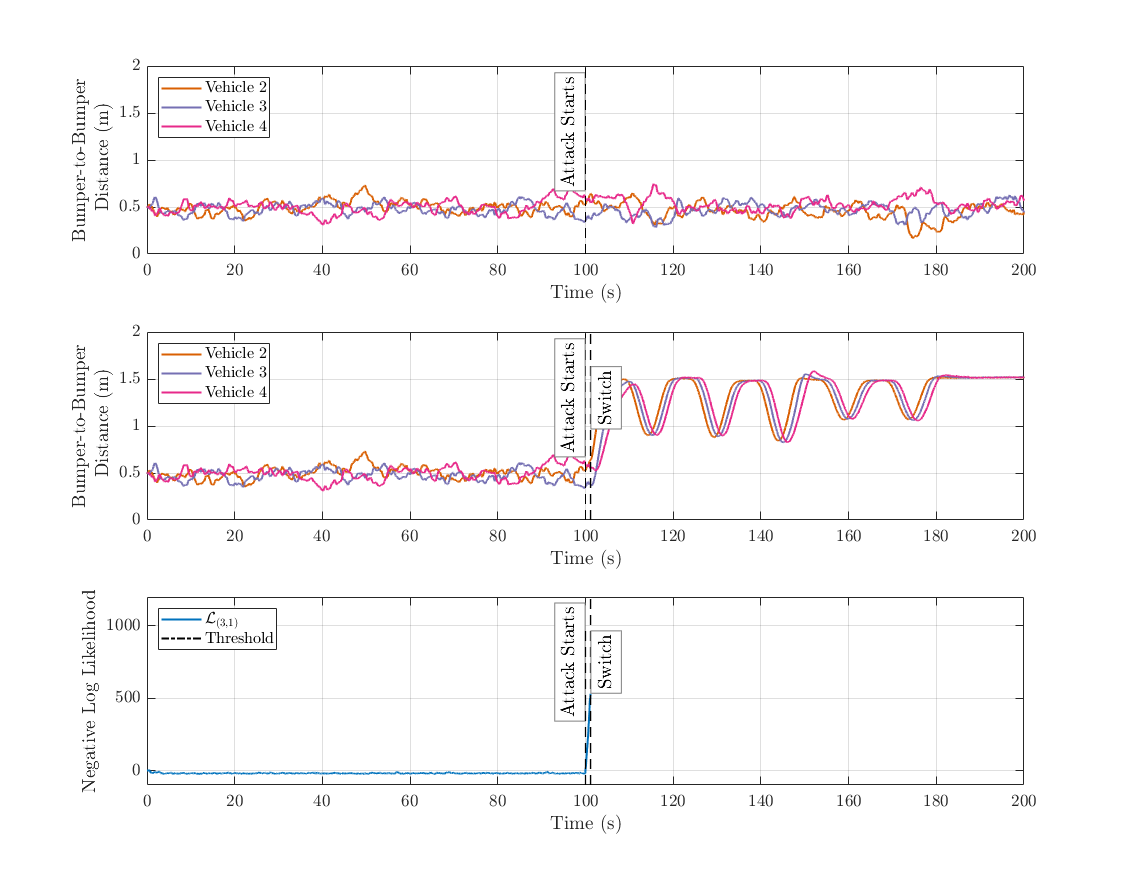}
    \caption{
    (top) Performance of the level 2 controller after a replay attack without switching to the level 1 controller.
    However, it completes the trajectory without crashing.
    (middle) Platoon switching to level 1 controller after detecting the attack and completing the entire trajectory. 
    (bottom) Negative log likelihood of channel which detected the replay attack first.
    }
    \label{fig:results_lf_replay}
\end{figure*}

\end{document}